\documentclass[11pt]{article}
\pdfoutput=1

\usepackage[nottoc]{tocbibind}
\usepackage[utf8]{inputenc}
\usepackage{graphicx}
\usepackage{float}
\usepackage{bm}
\usepackage{cite} 
\usepackage{amsthm}
\usepackage{amsmath}
\usepackage{amssymb}
\usepackage{mathrsfs}
\usepackage{dsfont}
\usepackage[all]{xy}
\usepackage{tikz-cd}
\usepackage{stmaryrd}
\usepackage{enumitem}
\usepackage{tabularx}
\usepackage{placeins}
\usepackage{hyperref}
\newcommand{\bea}{\begin{eqnarray}}
\newcommand{\eea}{\end{eqnarray}}
\newcommand{\be}{\begin{equation}}
\newcommand{\ee}{\end{equation}}
\newcommand{\ba}{\begin{aligned}}
\newcommand{\ea}{\end{aligned}}

\usepackage{comment}

\numberwithin{equation}{section}
\newcommand\cB{\mathcal B}
\newcommand\cC{\mathcal C}

\newcommand\cZ{\mathcal Z}

\newcommand{\rB}{\mathrm B}

\newcommand\sC{\mathscr C}
\newcommand\sD{\mathscr D}

\newcommand{\Aut}{\mathscr{A}ut}

\newcommand{\Z}{\mathbb{Z}}
\newcommand{\cVect}{\mathbf{3Vect}}
\newcommand{\bVect}{\mathbf{2Vect}}
\newcommand{\Mod}{\mathbf{Mod}}
\newcommand{\TY}{\mathbf{TY}}
\newcommand{\Bimod}{\mathbf{Bimod}}
\newcommand{\Pic}{\mathscr{P}ic}
\newcommand{\BrPic}{\mathscr{B}r\mathscr{P}ic}
\newcommand{\Witt}{\mathcal{W}itt}

\theoremstyle{definition}
\newtheorem{Definition}{Definition}[subsection]

\theoremstyle{plain}
\newtheorem{Theorem}[Definition]{Theorem}

\theoremstyle{plain}

\theoremstyle{plain}
\newtheorem{Proposition}[Definition]{Proposition}

\theoremstyle{plain}
\newtheorem{Lemma}[Definition]{Lemma}

\theoremstyle{plain}
\newtheorem{Corollary}[Definition]{Corollary}

\theoremstyle{plain}
\newtheorem{Conjecture}[Definition]{Conjecture}

\theoremstyle{plain}

\theoremstyle{plain}

\theoremstyle{definition}

\theoremstyle{definition}
\newtheorem{Example}[Definition]{Example}

\theoremstyle{remark}
\newtheorem{Remark}[Definition]{Remark}

\theoremstyle{plain}

\newtheorem*{notation*}{Notation}
\newtheorem*{Assumption*}{Assumption}

\title{\huge
Fusion 3-Categories for Duality Defects\\
}

 \author{Lakshya Bhardwaj$^1$, Thibault Décoppet$^2$, Sakura Sch\"afer-Nameki$^1$, Matthew Yu$^1$\\
 \smallskip\\
 {\it $^1$Mathematical Institute, University of Oxford, OX2 6GG, Oxford, UK }\\
  {\it $^2$Department of Mathematics, Harvard University, Cambridge, MA 02138, USA }\\
 }
 \date{}

\begin{document}

\maketitle
    \hspace{1cm}
    \begin{abstract}
    \noindent
We study the fusion 3-categorical symmetries for quantum theories in (3+1)-dimensions with  self-duality defects. Such defects have been realized physically by half-space gauging in theories with one-form symmetries $A[1]$ for an abelian group $A$, and have found applications in the continuum and the lattice. These fusion 3-categories will be called (generalized) Tambara-Yamagami fusion 3-categories ({\bf 3TY}), in analogy to the TY fusion 1-categories.
    We consider the Brauer-Picard and Picard 4-groupoids to construct these categories using a 3-categorical version of the extension theory introduced by Etingof, Nikshych and Ostrik. 
    These two 4-groupoids correspond to looking at the construction of duality defects either directly from the 4d point of view, or from the point of view of the 5d Symmetry Topological Field Theory (SymTFT). 
    At this categorical level, the Witt group of non-degenerate braided fusion 1-categories naturally appears in the aforementioned 4-groupoids and represents enrichments of standard duality defects by (2+1)d TFTs. 
    Our main objective is to study graded extensions of the fusion 3-category $\mathbf{3Vect}(A[1])$ for some finite abelian group $A$, which is the symmetry category associated to a (3+1)d theory with 1-form symmetry $A$. Firstly, we do so explicitly using invertible bimodule 3-categories and the Brauer-Picard 4-groupoid. Secondly, we use that the Brauer-Picard 4-groupoid of $\mathbf{3Vect}(A[1])$ can be identified with the Picard 4-groupoid of its Drinfeld center. Moreover, the Drinfeld center of $\mathbf{3Vect}(A[1])$, which represents topological defects of the SymTFT, is completely described by a sylleptic strongly fusion 2-category formed by topological surface defects of the SymTFT. These are classified by a finite abelian group equipped with an alternating 2-form. We relate the Picard 4-groupoid of the corresponding braided fusion 3-categories with a generalized Witt group constructed from certain graded braided fusion 1-categories using a twisted Deligne tensor product. In tractable examples, we are able to carry out explicit computations so as to understand the categorical structure of the $\mathbb{Z}/2$ and $\mathbb{Z}/4$ graded {\bf 3TY} categories.
    \end{abstract}

    \newpage

\tableofcontents

\section{Introduction}

The study of higher-form and categorical symmetries
 has greatly impacted our understanding of quantum theories. For reviews of generalized symmetries see  \cite{Schafer-Nameki:2023jdn, Bhardwaj:2023kri,Shao:2023gho, Brennan:2023mmt, 
 Luo:2023ive, Gomes:2023ahz}. In particular, fusion 1-categories have  found many applications in describing the symmetries in (1+1)d theories \cite{Bhardwaj:2017xup,Chang:2018iay,Thorngren:2019iar, Yu:2020twi,Thorngren:2021yso,inamura2021topological,Huang:2021zvu,Inamura:2022lun,Chatterjee:2022jll,Zhang:2023wlu,Huang:2023pyk,Bhardwaj:2023fca,Bhardwaj:2023idu,Bhardwaj:2023bbf,Bhardwaj:2024qrf, Bhardwaj:2024kvy,Bhardwaj:2024ydc,Bhardwaj:2024wlr, Chatterjee:2024ych} and fusion 2-categories for applications to (2+1)d theories
\cite{Roumpedakis:2022aik,Bhardwaj:2022yxj,Bhardwaj:2022lsg,Bartsch:2022mpm,Bhardwaj:2022kot,Bhardwaj:2022maz,Bartsch:2022ytj,Delcamp:2023kew,Inamura:2023qzl,DY23,Bhardwaj:2023fca,Choi:2024rjm,Bhardwaj:2024qiv}, and (3+1)d topological theories \cite{lan2018classification,Kong:2020wmn,KW:2020,KW2:2020jne,Barkeshli:2022wuz,Zhao:2022yaw}.

Going to higher dimensions and working with higher-dimensional operators leads to many technical complications on the categorical side. One quickly realizes that there is a proliferation of operators which themselves can host theories described by a fusion 1-category. 
However, developments in physics have pushed progress to (3+1)d and fusion 3-categories with numerous concrete realizations  and implications \cite{Choi:2021kmx, Kaidi:2021xfk,Bhardwaj:2022yxj,Choi:2022zal, Bhardwaj:2022lsg,Antinucci:2023ezl,Cordova:2023bja, Antinucci:2024ltv}. 
A particularly simple class of examples of higher categories are given by the categories describing (3+1)d theories with a non-invertible duality symmetry \cite{Choi:2021kmx,Kaidi:2021xfk}. These have also been realized on the lattice in \cite{Koide:2021zxj, Gorantla:2024ocs}.
These duality symmetries in (3+1)d generalize the Kramers-Wannier duality in (1+1)d for the critical Ising CFT. 
Constructing fusion 1-categories exhibiting a duality symmetry in (1+1)d can be done using extension theory as developed in \cite{ENO2}. The resulting categories are the so-called Tambara-Yamagami (TY) categories \cite{TY}. 

The goal of this paper is to explore fusion 3-categories that characterize such duality symmetries, developing the mathematical framework. 
We exploit extension theory for fusion 3-categories with the goal to study the categories behind the $\Z/2$ and $\Z/4$ duality symmetry for (3+1)d theories with a finite one-form symmetry studied in \cite{Kaidi:2021xfk,Kaidi:2022cpf,Antinucci:2023ezl,Cordova:2023bja}. The most elementary example of such theories include the free $U(1)$ gauge theory and 
$\mathcal{N}=4$ super Yang-Mills \cite{Bashmakov:2022jtl,KZZ:2022uux,Antinucci:2022cdi,Lawrie:2023tdz,Sela:2024okz}, where the self-duality is obtained by gauging a 1-form symmetry. 
We are able to unpack more of the categorical structure for symmetries of (3+1)d theories with a duality than what has previously appeared in the literature. In addition, we also compute the fusion rules of objects for the corresponding fusion 3-categories without invoking background fields. 

The ability to stack 3-dimensional operators with topological theories enriched with a symmetry makes studying self-duality defects rich in (3+1)d and we explore the effects in detail. In order to keep track of the different theories up to equivalence, we need to introduce a new generalized Witt group of graded braided fusion 1-categories with twisted product. We expect that these generalized Witt groups may find applications in the study of symmetry enriched topological theories in (2+1)d \cite{Barkeshli:2014cna}, where the symmetry provides the grading. There have also been applications where the symmetry enrichment is due to a continuous group \cite{Cheng:2022nds}.

Finally, motivated physically by the prospect of understanding the (3+1)d theory, the natural category to consider is the (Drinfeld) center of the graded fusion 3-category, i.e. the Symmetry Topological Field Theory (SymTFT) for the boundary symmetry category \cite{Ji:2019jhk,Apruzzi:2021nmk,Freed:2022qnc}, which as shown in \cite{Bhardwaj:2023ayw} encodes not only the symmetry but also the generalized charges \cite{Lin:2022dhv, Bhardwaj:2023wzd, Bartsch:2023pzl}.
As an intermediary step towards this 3-category, we define crossed braided fusion 3-categories and explain how they can be constructed via extension theory using the Picard 4-groupoid, which is described by the generalized Witt group. Understanding its structure and, in particular, its generators is essential for building these new 3-categories.
The center itself is then the equivariantization of the crossed braided 3-category. This last step is challenging to perform, and we thus leave it for future work.

\subsubsection*{Mathematical Content}

In order to construct new fusion 1-categories out of the ones that are already known, extension theory, as originally introduced in \cite{ENO2}, is one of the most useful methods. For instance, it can be used to give a conceptual construction of the so-called Tambara-Yamagami categories \cite{TY}, which are $\mathbb{Z}/2$-graded extensions of the pointed fusion 1-categories $\mathbf{Vect}(A)$ for some finite abelian group $A$ by $\mathbf{Vect}$. Extension theory can be generalized to higher fusion categories:\ The case of fusion 2-categories was dealt with in \cite{D11} using an argument from \cite{JFR3} that applies to all higher fusion categories simultaneously. Bypassing extension theory, fusion 2-categories with a Tambara-Yagami defect, that is $\mathbb{Z}/2$-graded fusion 2-categories with non-trivially graded component given by $\mathbf{2Vect}$, were classified in \cite{DY23} using the fact the such fusion 2-categories are automatically group-theoretical. In the case of fusion 3-categories however, the use of extension theory is unavoidable.

As is familiar from the theory of fusion 1-categories, in order to study extensions of a fusion 3-category, one needs to understand the associated Brauer-Picard 4-groupoid of invertible bimodule 3-categories. We will focus on the fusion 3-category $\mathbf{3Vect}(A[1])$ associated to a finite abelian group $A$. This fusion 3-category is connected so that its structure is completely determined by the associated braided fusion 2-category $\Omega\mathbf{3Vect}(A[1]) = \mathbf{2Vect}(A[0])$ of $A$-graded 2-vector spaces. Said differently, we have $\mathbf{3Vect}(A[1])\simeq \mathbf{Mod}(\mathbf{2Vect}(A[0]))$. We are able to carry out some explicit computations using this perspective. In particular, with $A=\mathbb{Z}/n$, the invertible $\mathbf{3Vect}(A[1])$-bimodule 3-categories built from $\mathbf{3Vect}$ yield elements of order $4$ in the Brauer-Picard group if $n>2$ and of order $2$ if $n=2$. This means that, for $A=\mathbb{Z}/n$ with $n>2$, the corresponding extensions of $\mathbf{3Vect}(A[1])$ by $\mathbf{3Vect}$ are graded by $\mathbb{Z}/4$. This last fact is well-known in the Physics literature \cite{Kaidi:2022cpf}. We will refer to these categories as extended Tambara-Yamagami 3-categories, highlighting the fact that they are $\Z/4$-graded.

It is also possible to take a different perspective. Namely, one can leverage the fact that the Brauer-Picard 4-groupoid is equivalent to the Picard 4-groupoid, that is the space of invertible module 3-categories of the Drinfeld center. Now, the Drinfeld center of $\mathbf{3Vect}(A[1])$ is a connected braided fusion 2-category. Specifically, it is of the form $\mathbf{Mod}(\mathfrak{S})$ for the nondegenerate sylleptic strongly fusion 2-category $\mathfrak{S}$ corresponding to the group $A\oplus \widehat{A}$ equipped with the canonical alternating 2-form. Slightly more generally, recall from \cite{JFY2} that, given a finite abelian group $B$ equipped with an alternating 2-form $s$, there is a corresponding sylleptic strongly fusion 2-category $\mathbf{2Vect}^s(B[0])$, and an associated braided fusion 3-category $\mathbf{Mod}(\mathbf{2Vect}^s(B[0]))$. We show that the Picard 4-groupoid of the fusion 3-category $\mathbf{Mod}(\mathbf{2Vect}^s(B[0]))$ is a variant of the Witt group of non-degenerate braided fusion 1-categories from \cite{DMNO}. More precisely, the Witt group is precisely the Picard 4-groupoid of $\mathbf{3Vect}$ by \cite{BJSS} as non-degeneracy of a braided fusion 1-category is a higher categorical invertibility condition. We identify the Picard 4-groupoid of $\mathbf{Mod}(\mathbf{2Vect}^s(B[0]))$ with a generalized Witt group constructed from suitably invertible $B$-graded braided fusion 1-categories with product given by the Deligne tensor product twisted by $s$. We also consider the subgroup generated by (suitably invertible) graded braided fusion 1-categories that are pointed. This subgroup is more amenable to calculations, and we carry out some explicit computations in the case $B=\mathbb{Z}/2\oplus\mathbb{Z}/2$ with non-trivial alternating 2-form, which can be identified with the Picard group of the Drinfeld center of $\mathbf{3Vect}(\mathbb{Z}/2[1])$.

Finally, we obtain a 3-categorical version of another classical result from \cite{ENO2}. Namely, we introduce crossed braided fusion 3-categories, and show that such fusion 3-categories are classified by maps into the Picard 4-groupoid of the braided fusion 3-category corresponding to the trivially graded factor. But, as the Brauer-Picard 4-groupoid is equivalent to the Picard 4-groupoid of the Drinfeld center, any group-graded extension of a fusion 3-category determines a crossed braided extension of its Drinfeld. Categorifying a result from \cite{GNN}, we assert that the equivariantization of this crossed braided fusion 3-category coincides with the Drinfeld center of the group-graded fusion 3-category. Using our previous computations in the generalized Witt group associated to $\mathbb{Z}/2\oplus\mathbb{Z}/2$ with non-trivial alternating 2-form, we partially describe $\mathbb{Z}/2$-crossed braided extensions of the Drinfeld center of $\mathbf{3Vect}(\mathbb{Z}/2[1])$.

\subsubsection*{Structure}

In section \ref{section:physics}, we give physical interpretations of all the mathematical tools that we will work with. In section \ref{section:highercats}, we recall some higher categorical notions, especially concerning higher Morita categories, as well as useful facts about fusion 2-categories that will be used throughout. In section \ref{section:extensiontheory}, we spell out extension theory for fusion 3-categories using the Brauer-Picard 4-groupoid. We include examples of $\Z/2$- and $\Z/4$-graded extensions exhibiting duality defects. In section \ref{section:crossedbraided}, we study the Picard space of a braided fusion 3-category and define crossed braided fusion 3-categories. Associated to a finite abelian group equipped with an alternating 2-form, we introduce a generalized Witt group, which agrees with the Picard group of the corresponding braided fusion 3-category. We perform some explicit computations for the group $\mathbb{Z}/2\oplus\mathbb{Z}/2$ equipped with the non-trivial alternating 2-form. Finally, in section \ref{section:outlook}, we explain some of the computational challenges that we are facing when working with fusion 3-categories.

\subsection*{Acknowledgments}
It is a pleasure to thank Theo Johnson-Freyd, Lukasz Fidkowski, Dmitri Nikshych, and Xingyang Yu for helpful conversations and comments on this paper. The author L.B. is funded as a Royal Society University Research Fellow through grant URF{\textbackslash}R1\textbackslash231467. The author T.D. is sponsored by the Simons Collaboration on Global Categorical Symmetries.
The authors S.S.-N. and M.Y. are supported by the EPSRC Open Fellowship EP/X01276X/1.
No authors have competing interests to declare that are relevant to the content of this article. There is no data available for this article to declare.

%%%%%%%%%%%%%%%%%%%%%%%%%%%%%%%%%%%%%%%%%%%
\section{Physical Context and Motivation}\label{section:physics}
%%%%%%%%%%%%%%%%%%%%%%%%%%%%%%%%%%%%%%%%%%%

%%%%%%%%%%%%%%%%%%%%%%%%%%%%%%%%%%%%%%%%%%%%%%%%%%%%%%
\subsection{Physical Context for Tambara Yamagami 3-Categories}\label{subsection:3TYphysics}
%%%%%%%%%%%%%%%%%%%%%%%%%%%%%%%%%%%%%%%%%%%

Before diving into the mathematical analysis, we outline the physics framework for the mathematical structures that will be considered. The setting is that of 
a (3+1)d theory $\mathcal{T}$ with a one-form symmetry $A[1]$ given by the fusion 3-category $\mathscr{C} = \cVect({A[1]})$. The brackets denote the degree in which the group lives and gives the ``form number" of the symmetry that the group describes. That is, for any non-negative integer $n$ and with $A$ an abelian group, if $n \geq 1$, we use $A[n]$ to denote the higher group whose classifying space is the
Eilenberg-MacLane space $K(A,n+1)$. We will sometimes drop the brackets for the group when the degree is 0, i.e.\ it describes a 0-form symmetry. When taking the group cohomology for a (higher) group, we always mean the cohomology of its classifying space. 

The theory $\mathcal{T}$ is self-dual if the gauged theory $\mathcal{T}/A[1]$ is isomorphic to $\mathcal{T}$, and the isomorphism is symmetric with respect to gauging. We aim to classify the self-dual symmetries of $\mathcal{T}$. This can be  studied  directly from either the perspective of the QFT $\mathcal{T}$, where we study categorical extensions of $\mathscr{C}$, or from the SymTFT perspective. From the latter perspective, which we will refer to as the bulk point of view,  the starting point is a (4+1)d topological theory on an interval. One of the boundaries is gapped, and hosts the symmetry category $\mathscr{C}$. Categorically speaking, the topological defects of the SymTFT are given by $\cZ(\mathscr{C})$, the Drinfeld center of $\mathscr{C}$.  For our purposes, we can realize the SymTFT as a Dijkgraaf-Witten (DW) theory for the 1-form symmetry, 
whose action  is given by 
\be
S = \frac{1}{n}\int b^{(2)}\delta c^{(2)}\,,
\ee
where $n$ is the order of $A= \Z/n$. Here $b^{(2)}$ and $c^{(2)}$ are two dynamical gauge fields for $A[1]$. There is a global symmetry $G$ that performs the interchange 
\be
G:\qquad b^{(2)}\to c^{(2)}\,,\qquad c^{(2)} \to -b^{(2)} \,,
\ee
When the group $A= \Z/2$ then the global symmetry is given by $\Z/2$ and when $n >2$ then the global symmetry is given by $\Z/4$.
We will study the self-dual theories, whose full symmetry is described by an {\bf extended Tambara-Yamagami fusion 3-category}. Such a category is described from the mathematical point of view as arising from $G$-graded extensions of $\mathscr{C}$. On the side of the bulk, extension theory allows us to consider a 3-category that is intimately related to the centers of the $G$-graded extensions of $\mathscr{C}$.
The passage between the boundary and bulk
is summarized by the square in Figure \ref{fig:gauging}.

\begin{figure}
$$
\begin{tikzpicture}
\draw [->] (0,2) -- (4,2);  
\draw[->] (-0.3,1.7) -- (-0.3,0.3) ; 
\draw[->] (4.5,1.7) -- (4.5,0.3) ; 
\draw [->] (0,0) -- (1.9,0); 
\draw [->] (2.5,0) -- (4,0); 
\node[left] at (0,2) {$\sC$}; 
\node[right] at (4,2) {$\ \ \sD$}; 
\node[right] at (4,0) {$\mathcal{Z}(\sD)$}; 
\node[left] at (0,0) {${\mathcal{Z}(\sC)}$} ;
\node at (2.2,0) {${\mathscr{B}}$};
\node[left] at (-0.3,1) {\footnotesize{Center}};
\node[right] at (4.5,1) {\footnotesize{Center}};
\node[above] at (2,2) {\footnotesize{Extension}};
\node[above] at (1,0) {\footnotesize{Extension}};
\node[above] at (3.2,0) {\footnotesize{Equiv.}};
\end{tikzpicture}
$$
\caption{Mathematical Implementation of Gauging of the SymTFT: $\mathscr{C}$ is the fusion 3-category e.g. $ \cVect({A[1]})$. Incorporating the duality symmetry (``Extension'') results in an extended Tambara-Yamagami (TY) fusion 3-category $\mathscr{D}=\mathbf{3TY}(\cVect({A[1]}))$. Alternatively we can consider the SymTFT, and the associated topological defects given by the Drinfeld center $\mathcal{Z} (\mathscr{C})$ of $\mathscr{C}$. Including and then gauging the 0-form symmetry of the center (``Extension" and ``Equivariantization") results in the center $\mathcal{Z}(\mathscr{D})$ of $\mathscr{D}$. }
\label{fig:gauging}
\end{figure}

The description of gauging on the side of the bulk, that is, for the SymTFT, is more involved, and may be summarized as follows:
\begin{itemize}
    \item We begin with the 5d SymTFT given by the Dijkgraaf-Witten theory for $A[1]$, together with a non-anomalous action of a group $G$. 
    \item We gauge the $G$ symmetry to produce the SymTFT of another boundary theory. The gauging is mathematically described by a two step process that involves an \textit{extension} step followed by an \textit{equivariantization} step. 
    \item The data that is needed for the extension step in the SymTFT is the same data that one needs for the extension on the boundary. As for the equivariantization step, it requires no additional data. Further, the diagram in Figure \ref{fig:gauging} commutes.
\end{itemize}

Extensions of the (4+1)d SymTFT graded by the finite group $G$ are classified by homotopy classes of maps from $\mathrm{B}G$ into the Picard groupoid $\mathrm{B}\Pic(\mathcal{Z}(\sC))$. But, there is a canonical map $\mathrm{B}\Pic(\mathcal{Z}(\sC))\rightarrow \mathrm{B}\mathscr{A}ut^{br}(\cZ(\mathscr{C}))$ to the braided automorphism of the center of $\mathscr{C}$. Therefore, any extension gives rise to \be\label{rhoGAut}
\rho: \quad G \rightarrow \Aut^{{br}}(\cZ(\mathscr C))\,,
\ee
which describes the action of a $G$ 0-form symmetry of the SymTFT $\mathfrak{Z}(\mathscr{C})$ on its topological defects forming the center $\mathcal{Z}(\mathscr{C})$. One of the major difficulties that we have faced is that, unlike in the case of fusion 1-categories, the map $\mathrm{B}\Pic(\mathcal{Z}(\sC))\rightarrow \mathrm{B}\Aut^{br}(\cZ(\mathscr{C}))$ is not an equivalence on low degree homotopy groups. Consequently, it is much more difficult to understand $\Pic(\mathcal{Z}(\sC))$. 

The global symmetry of the bulk needs to extend consistently to the boundary of the SymTFT; from this point of view, the SymTFT provides the order of the duality symmetry that one should expect on the boundary after gauging. Mathematically, this fact is captured by the abstract equivalence of $\mathrm{B}\Pic(\cZ(\sC))$ and $\mathrm{B}\BrPic(\sC)$, where we think of the latter groupoid as parametrizing the codimension-1 topological defects on the boundary attached to codimension-1 topological defects of the bulk SymTFT. Even though there is an equivalence, explicitly describing this equivalence is challenging. $\BrPic(\sC)$ is the so-called Brauer-Picard space -- these are the invertible $\mathscr{C}$-bimodule 3-catgories. More precisely, $G$-graded extensions of $\sC$ are classified by homotopy classes of maps from $\mathrm{B}G$ into the Brauer-Picard 4-groupoid $\mathrm{B}\BrPic(\sC)$. This will be discussed in section \ref{section:extensiontheory}.

For the purposes of studying extensions and therefore duality symmetries for the boundary theory using the SymTFT, it is crucial that we understand the Picard 4-groupoid of $\mathcal{Z}(\mathscr{C})$, and we can use what is known about topological theories in $(4+1)$d to make this 4-groupoid explicit in the cases that are most physically relevant. 
%\thib{I think I had made this comment previously, I do not think that the statement you are referring to applies in this case. Your result with Theo is about the classification of non-degenerate braided fusion 3-categories up to Morita equivalence. We are looking at (braided) fusion 3-categories up to monoidal equivalence. Those are very different things. If you want to bring in your result, you should do it differently.} In particular, topological theories, following the definition in \cite{JF}, are classified in (4+1)d by sylleptic strongly fusion 2-categories $\mathfrak{S}$. We will use $\mathfrak{S}$ to study the category $\cZ(\mathscr{C})$, and in particular to get a handle on the topological operators in the SymTFT.

We now describe some properties we expect of the 2-category that describes the surface defects in the $(4+1)$d SymTFT. 
Being in $(4+1)$d the 2-category is naturally braided, in fact sylleptic. The term syllepsis is less familiar, but  essentially corresponds to the linking between topological defects in the (4+1)d SymTFT. Let us review this in the context of a DW-theory, which is the SymTFT for a 1-form symmetry in (3+1)d. The linking of two surfaces $\mathrm{link}(M_2,M'_2)$ constructed by $b^{(2)}$ and $c^{(2)}$ respectively has the property that it is antisymmetric
\begin{equation}
    \mathrm{link}(M_2,M'_2) = -\mathrm{link}(M'_2,M_2)\,.
\end{equation}
Braiding two surfaces results in 
\begin{equation}
    M_2 M'_2 = e^{\left({ 2\pi i/n\, \mathrm{link}(M_2,M'_2)}\right)} M'_2 M_2\,,
\end{equation}
where the phase associated to the linking will be called \textit{syllepsis}, and later on be denoted by $s$. Here we have specified the phase for $A=
\mathbb{Z}/n$. This is extra structure that is equipped onto our 2-category, enhancing the braiding. Let us write $\mathbf{2Vect}^s(A\oplus A)$ for the fusion 2-category of $A\oplus A$-graded 2-vector spaces with syllepsis $s$. The braided 3-category $\mathcal{Z}({\sC})$ describing the SymTFT can then be identified with $\Mod(\mathbf{2Vect}^s(A\oplus A))$, the braided 3-category of finite semisimple modules for $\mathbf{2Vect}^s(A\oplus A)$. The fact that the SymTFT arises in this way implies that the 4-dimensional operators of the DW-theory which implement the global $G$-symmetry are condensation defects and therefore can end on a 3-dimensional twist defects in the bulk. 

The twist defects are captured by the space $\mathscr{P}ic(\mathcal{Z}(\sC))=\mathscr{P}ic(\mathbf{Mod}(\mathbf{2Vect}^s(A\oplus A)))$, which we will unpack in more mathematical detail in section section \ref{section:crossedbraided}. A new feature that does not appear in lower dimensions is the ability to stack on 3d TFTs to the submanifold where the twist defect lies. 
Such 3d TFTs are classified by non-degenerate braided fusion categories. It is useful to consider these up to centers, because centers can be condensed away to the trivial TQFT. The Witt group introduced in \cite{DMNO} describes the structure of non-degenerate braided fusion categories up to Drinfeld centers. Physically speaking, any two non-degenerate braided
fusion 1-categories with distinct Witt classes in the Witt group are not related by anyon condensation, and give rise to twist defects not related by gauging. For our present purposes, we need to consider a generalized Witt group, which classifies $s$-invertible $A\oplus A$-graded fusion 1-categories up to a certain equivalence relation. 
Namely, the generalized Witt group $\Witt(A\oplus A,s)$
can be used to model $Pic(\mathbf{Mod}(\mathbf{2Vect}^s(A\oplus A)))$. %This is genuinely a new effect in this higher-dimensional setting, which makes the study of these structures interesting and much richer than in lower dimensions.

\paragraph{{\bf 3TY} and Generalized {\bf 3TY} Categories.}
    We mainly address extension theory for fusion 3-categories and fusion 1-categories in this paper, with the most prominent examples being  the Tambara-Yamagami (TY) fusion 3-category, 3TY and the Tambara-Yamagami fusion 1-category. The way 3TY categories are defined parallels the definition that is familiar in the context of 1-categories. More  precisely, {\bf 3TY} categories are $\mathbb{Z}/2$ graded and $\mathbb{Z}/4$ graded 3-categories will be referred to as generalized {\bf 3TY} categories \footnote{1TY categories are exclusively $\Z/2$-graded by definition}. The trivially graded component has 1-form symmetry $A[1]$, and a non-trivially graded component with an object that enacts the duality symmetry.  
    It is also possible to realize 2TY, as the symmetry for (2+1)d theories and this was done in \cite{DY23}. The techniques used there were much different than the ones presented here. More precisely, it was shown that all 2TY categories are group theoretical, thereby providing a lot of leverage in their study. It is well-known that not all TY categories are group-theoretical \cite{GNN}. Likewise, it seems overwhelmingly likely that not all 3TY categories are group-theoretical.

%%%%%%%%%%%%%%%%%%%%%%%%%%%%%%%%%%%%%%%%%%%%%%%%%%%%%
\subsection{$\Z/2$-graded Extension and TY Fusion 1-categories}
%%%%%%%%%%%%%%%%%%%%%%%%%%%%%%%%%%%%%%%%%%%%%%%%%%%%%
While our main goal is to develop extension theory to construct $\Z/2$ and $\Z/4$-graded 3-categories, we find it instructive to review how extension theory is used to construct Tambara-Yamagami (TY) fusion 1-categories. In particular, we review the Brauer-Picard 2-groupoid and how it makes the data needed to classify TY fusion-categories apparent. When we tackle the 3-category case, we will be more easily able to give the data of TY fusion 3-categories. 

A Tambara-Yamagami fusion 1-category initially introduced in \cite{TY} and denoted by $\TY(A,\chi, \tau)$ is a $\Z/2$-graded extension of $\mathbf{Vect}(A[0])$ by $\mathbf{Vect}$. We therefore  need to understand maps from $\rB\mathbb{Z}/2$
to $\rB\mathscr{B}r\mathscr{P}ic(\mathbf{Vect}(A[0]))$, the Brauer-Picard space of $\mathbf{Vec}(A[0])$. This space $\BrPic(\mathbf{Vect}(A[0]))$ has homotopy groups given by:
\be 
\begin{tabular}{|c|c|c|c|c|c|}
\hline
$\pi_0$ & $\pi_1$ & $\pi_2$ \\
\hline \\[-1em]
$BrPic(\mathbf{Vect}(A[0]))$ & $A\oplus \widehat{A}$  & $\mathbb{C}^{\times}$\\
\hline
\end{tabular}
\ee
where $BrPic(\mathbf{Vect}(A[0]))$ is the group of invertible bimodules for $\mathbf{Vect}(A[0])$ and $\widehat{A}$ is the Pontryagin dual group to $A$. We note that the maps from $\Z/2$ to $BrPic(\mathbf{Vect}(A[0]))$ that produce a Tambara-Yamagami fusion 1-category are determined by an invertible $\mathbf{Vect}(A[0])$-bimodule structure on $\mathbf{Vect}$. Since $BrPic(\mathbf{Vect}(A[0])) \cong \Aut^{br}(\mathcal{Z}(\mathbf{Vect}(A[0])))\cong O(A\oplus \widehat A)$  we explicitly look at maps that pick a $\Z/2$ element from $O(A\oplus \widehat{A})$ which implement the swap action in the DW theory. From the point of view of the boundary theory which now has a self-duality, the duality defect is represented by the simple object in $\mathbf{Vect}$.
It was shown in \cite[Proposition 9.3]{ENO2} that an invertible $\mathbf{Vect}(A[0])$-bimodule category of order two based on $\mathbf{Vect}$ is completely determined by a 
choice of symmetric non-degenerate bicharacter $\chi$ on $A$ that gives an identification $A \cong \widehat{A}$.
Next, we consider how to extend the map $\mathrm{B}\mathbb{Z}/2\rightarrow \mathrm{B}BrPic(\mathbf{Vect}(A[0]))$ to a map $\mathrm{B}\mathbb{Z}/2\rightarrow \mathrm{B}\tau_{\leq 1}\BrPic(\mathbf{Vect}(A[0]))$. Observe that the group $\mathbb{Z}/2$ acts on $\pi_1(\rB\BrPic(\mathbf{Vect}(A[0])))= A \oplus \widehat{A}\cong \mathrm{Inv}(\mathcal{Z}(\mathbf{Vect}(A[0])))$. In the specific case of Tambara-Yamagami, the action is the swap action between the two factors in $A \oplus \widehat{A}$ (the identification between $A$ and $\widehat{A}$ being induced by $\chi$). As explained in detail in \cite{ENO2}, the first obstruction to a tensor structure on the graded category is given by a class in $H^3(\Z/2[0]; \underline{A\oplus \widehat{A}})$, where $\underline{A\oplus \widehat{A}}$ denotes that the coefficients are twisted by the swap action. This group evaluates to zero, which means that there are $H^2(\Z/2[0]; \underline{A\oplus \widehat{A}})$ choices of lifts. Since this latter group is also zero, we discover that there is actually no choice at all. Finally, extending the map $\mathrm{B}\mathbb{Z}/2\rightarrow \mathrm{B}\tau_{\leq 1}\BrPic(\mathbf{Vect}(A[0]))$ to a map $\mathrm{B}\mathbb{Z}/2\rightarrow \mathrm{B}\BrPic(\mathbf{Vect}(A[0]))$ corresponds to an associativity constraint obstruction class in $H^4(\Z/2[0];\mathbb{C}^\times)$, a group that vanishes, so that there are $H^3(\Z/2[0];\mathbb{C}^\times) \cong \Z/2$ choices of lift. This last choices gives the data of $\tau$ that features in the classical description of Tambara-Yamagami 1-categories.

We now consider the centers of TY categories, and, more generally, how the center of a $G$-graded fusion 1-category is related to the center of the trivially graded fusion 1-category. The centers of TY categories has been described in the literature in \cite{Izumi:2001mi,GNN}.
What is canonically referred to as gauging for braided fusion 1-categories starts with a braided fusion 1-category $\mathcal{B}$ and runs through the extension and equivariantization procedure in Figure \ref{fig:gauging} for a group $G$ giving the equivalence classes of braided tensor autoequivalences of $\mathcal{B}$ \cite{CGPW}. The map from $G$ means that the theory described by $\mathcal{B}$ has a $G$ global symmmetry, for which symmetry defects can be introduced. The group ${P}ic(\mathcal{B}) = \pi_0(\Pic(\mathcal{B}))$ of invertible module categories over $\mathcal{B}$ is the mathematical object describing the symmetry defects. We can see this gauging procedure in practice in the (2+1)d SymTFT associated to a (1+1)d theory with $\mathbf{Vect}(\Z/n)$ fusion category symmetry is a $\Z/n$-gauge theory described by $\mathcal{Z}(\mathbf{Vect}(\Z/n))$ with action $S=\frac{1}{n}\int b^{(1)} \delta c^{(1)}$. This theory has a $G= \Z/2$ symmetry taking $b^{(1)}\to c^{(1)}$ and $c^{(1)}\to b^{(1)}$, i.e.\ electromagnetic duality. Gauging the $\Z/2$ in the bulk gives the SymTFT for $\textbf{TY}[\Z/n]$ (see \cite[Section 4.1]{Kaidi:2022cpf}), which is just the center $\mathcal{Z}(\mathbf{TY}[\Z/n])$. Upon restricting to the boundary, i.e.\ forgetting from the center, we get the category $\textbf{TY}[\Z/n]$. Hence, we have an example for which the gauging procedure of $\Z/2$ taking place in the bulk results in the same category as the $\Z/2$-graded extension.
The SymTFTs for $G$-graded fusion 1-categories are $G$-gaugings of SymTFTs for the original fusion 1-category.
In the mathematical parlance, we started off with a fusion 1-category $\mathbf{Vect}({\Z/n})$ of which we took a $\Z/2$-graded extension to get $\mathbf{TY}[\Z/n]$. We naturally have the relative center $\mathcal{Z}_{\mathbf{Vect}(\Z/n)}(\mathbf{TY}[\Z/n])$ which is the category of $\mathbf{Vect}(\Z/n)$-bimodule functors from $\mathbf{Vect}(\Z/n) \to \mathbf{TY}[\Z/n]$. This has a canonical  $\Z/2$-crossed braided category structure \cite[Section 3.1]{GNN}, which according to the second step of the gauging procedure must be equivariantized. The result is equivalent as a braided fusion 1-category to $\mathcal{Z}(\mathbf{TY}[\Z/n])$. The twist defects that were discussed in section \ref{subsection:3TYphysics} for the (4+1)d picture are more easily seen in the (2+1)d SymTFT. 
We take $n=2$ and the symmetry in (1+1)d is $\mathbf{Vect}(\Z/2)$, with SymTFT $\mathcal{Z}(\mathbf{Vect}(\Z/2))$. The surface that implements a $\Z/2$ global symmetry is a condensation surface and can end on a line in the bulk which is a twist defect.
Furthermore it is also possible to twist the gauging by stacking with an invertible phase for the $\Z/2$ symmetry. Gauging the $\Z/2$ global symmetry in the bulk makes the twist defect a genuine line, and the SymTFT becomes
 $\mathcal{Z}(\TY[\Z/2,\chi,\tau])$. In \cite[Section 4]{Kaidi:2022cpf} the authors apply the two step gauging procedure to recover $\mathcal{Z}(\TY(A,\mathrm{triv},\mathrm{triv}))$, where $A=\Z/n$, but attempting to generalize their formalism to higher categories and making manifest all the data that is necessary to include in the gauging step is difficult. We hope that this paper provides a more transparent way to realize the data needed to define higher categories with duality defects, by taking an algebraic perspective.

\subsection{Fusion 3-Categories}
We will now summarize a few salient features of fusion 3-categories from a physical perspective, and also highlight aspects of these that do not appear in lower categories. 
We will always require the categories to be fusion (as opposed to multi-fusion), in particular the identity is a simple object and furthermore all objects are written as direct sums of simple objects. 
Broadly speaking a fusion 3-category characterizes topological defects of dimension 3 (objects) and lower (morphisms), thus acting as a symmetry of an at least (3+1)d QFT. 
As symmetries the objects of the category are associated to 0-form symmetries and the 1-, 2-, 3-morphisms are accordingly 1-, 2-, 3-form symmetries, respectively. For group-like symmetries of this type, the most general structure is a 4-group symmetry. More generally the topological defects can have ``non-invertible" fusion and at each level there is a notion of composition of $i$-morphisms which need not be strict if $i<3$.

Starting with the 1-category of finite dimensional vector spaces $\mathbf{Vect}$, one can consider $\Mod( \mathbf{Vect})$ which we take to be a one point delooping with respect to $\mathbf{Vect}$, as well as condensation completing. Applying this procedure once more gives $\mathbf{3Vect}$ where all the objects are condensation defects.
The important distinction compared to ``lower categories" is that general fusion 3-categories have objects that have nontrivial  (2+1)d TQFTs as endomorphisms \footnote{By nontrivial we mean that there is no anyon condensation to the vacuum.}. 
Furthermore, these categories can have multiple component, where in each component resides the objects that are related by anyon condensation.

To illustrate the previous points, let us consider a theory with a 1-form symmetry $A[1]$ in (3+1)d. This has topological defects that are surface operators $\mathbf{D}_2^{(g)}$, $g\in A$. However, the fusion 3-category $\cVect(A[1])$ has infinitely many three-dimensional topological defects given by $A$-graded fusion 1-categories, and we can consider condensation defects of the 1-form symmetry on any of these. 
This may seem trivial from the perspective of a symmetry acting on a (3+1)d QFT, however, we will see that such defects play an important role in the construction of the 3TY categories.

For many constructions in this work we need to study the module 3-categories associated to rigid algebras in $\cVect(A[1])$, see e.g.\ section \ref{section:highercats} and section \ref{section:extensiontheory}. In the sense of lower categories, these algebras are the condensible operators in our theory, and now should be thought of as categories themselves. For (2+1)d topological theories described by non-degenerate braided fusion categories, the algebras that are of most physical relevance are Lagrangian algebras, which are special connected étale algebras \footnote{A étale algebra in a braided fusion category $\cC$ is a commutative separable algebra. An étale algebra $A$ is connected if $\mathrm{dim} Hom_{\cC}(1,A) =1$}. Rigid and separable algebras in fusion 2-categories have been studied in \cite{D7}, and the physical implications for gauging surface operators is explained in \cite{DY22}.
We will work with fusion 3-categories, and consider module categories with respect to algebras constructed from $A$-graded braided fusion 1-categories. These algebras take the form $\Mod(\cB)$ where $\cB$ is an $A$-graded braided fusion 1-category. Given two $A$-graded braided fusion 1-categories $\cB_1$ and $\cB_2$, one can study the category of boundaries between them.  This is given by $\Bimod_{\mathbf{3Vect}(A[1])}(\Mod(\cB_1),\Mod(\cB_2))$
the 3-category of
the $A$-graded $\cB_1\text{-}\cB_2$-central multifusion 1-categories. By the folding trick, this is equivalent to finding boundaries for $\Mod(\cB_1)\boxtimes_{\mathbf{3Vect}(A[1])} \Mod(\cB^{rev}_2)$. The relevance for boundaries between algebras is explained further in section \ref{subsection:3Catduality} and the results are used for constructing duality defects in (3+1)d.

%%%%%%%%%%%%%%%%%%%%%%%%%%%%%%%%%%%%%%%%%%%%%%%%%%%%%
\section{Higher Categorical Setting}\label{section:highercats}
%%%%%%%%%%%%%%%%%%%%%%%%%%%%%%%%%%%%%%%%%%%%%%%%%%%%%

%\thib{Also note that the notion of a braiding has to appear fairly early on, so your discussion about levels of monoidality (currently at the end of this subsection) should also appear earlier. I will let you put it wherever is appropriate.}

In this section, we introduce the higher categorical objects that are going to be our main focus. We being by reviewing higher condensation monads and higher fusion categories. We then specialize our discussion to sylleptic fusion 2-categories and their corresponding braided fusion 3-categories. Finally, we argue that certain 4-categories associated to a sylleptic fusion 2-category are equivalent.

\subsection{Higher Condensations}

By an $n$-category, we mean a weak $n$-category in the sense that the composition of $k$-morphisms for $k< n$ is only well defined up to higher coherences, and the $n$-morphisms form a set. More precisely, following \cite{baez1998higher,JF}, weak $n$-categories can be defined iteratively as $(\infty,1)$-categories enriched in a weak $(n-1)$-category, where a weak $0$-category is a set. We will work in the $\mathbb{C}$-linear setting, where the $n$-morphisms form a vector space over $\mathbb{C}$. Physically, $n$-categories describe a collection of topological defects, along with a collection of sub-defects, of $n$ spacetime dimension. If there is no fusion or braiding structure on the $n$-category, then it describes a collection of topological boundary conditions of an $(n+1)$-dimensional\footnote{All dimensions mentioned in this work are spacetime dimensions.} TFT. If there are other structures on the $n$-category, then it describes other kinds of $n$-dimensional topological defects.

We begin by recalling the notion of an $n$-condensation in an $n$-category $\mathbf{C}$ from \cite{GJF}, which is an inductive definition. Let $X$ and $Y$ be objects of the 1-category. A $1$-condensation of $X$ onto $Y$ is a pair of morphisms $f:X\twoheadrightarrow Y$ and $g:X\hookrightarrow Y$ such that $f\circ g = \mathrm{id}_Y$, i.e.\ a split surjection of $X$ onto $Y$. An $n$-condensation is a pair of morphisms $f:X\rightarrow Y$ and $g:Y\rightarrow X$ together with an $(n-1)$-condensation $(f\circ g,\mathrm{id}_Y,\phi,\gamma,...)$ from $f\circ g$ onto $\mathrm{id}_Y$. We write $(X,Y,f,g,...)$ for an $n$-condensation. We denote by $\spadesuit_n$ the walking $n$-condensation, that is, the $n$-category with two objects $X$ and $Y$, such that functor $\spadesuit_n\rightarrow \mathbf{C}$ are precisely $n$-condensations.
 
We denote by $\clubsuit_n \subset \spadesuit_n$ the full subcategory on the object $X$. This subcategory has a unique 1-morphism $e := g \circ f$. Then, an $n$-condensation monad is a functor $(X,e,...):\clubsuit_n\rightarrow \mathbf{C}$. Said differently, $\clubsuit_n$ is the walking $n$-condensation monad. In slightly different terminology, we shall also that say that an $n$-condensation monad in an $n$-category $\mathbf C$ is an object $X \in \mathbf{C}$ together with a condensation algebra $(e,...)$ in the monoidal $(n-1)$-category $End_{\mathbf{C}}(X)$. In the case $n=1$, $1$-condensation monads are precisely idempotent morphisms of a 1-category. For $n=2$, $2$-condensation algebras are non-unital and non-counital special Frobenius algebras in the monoidal 1-category $End_{\mathbf{C}}(X)$.

We say that an $n$-condensation monad splits if it can be extended to an $n$-condensation. It was shown in \cite[Theorem 2.3.2]{GJF} that, if such an extension exists, it is unique up to a contractible space of choices. This leads to an inductive definition of condensation completeness. More precisely, we say that a $1$-category has all condensates if all 1-condensation monads, i.e.\ idempotents, splits. Then, we say that an $n$-category $\mathbf{C}$ has all condensates if all of its $Hom$-$(n-1)$-categories have all condensates, and every $n$-condensation monad in $\mathcal{C}$ splits. If $\mathbf{C}$ is a $n$-category, all of whose $Hom$-$(n-1)$-categories have all condensates, it can be canonically embedded into an $n$-category $Kar(\mathbf{C})$ that has all condensates. We refer the reader to \cite[Theorem 2.3.10]{GJF} for details, and only mention that the objects of $Kar(\mathbf{C})$ are $n$-condensation monads in $\mathbf{C}$, and the 1-morphisms are $n$-condensation bimodules as in \cite[Definition 2.3.3]{GJF}. In particular, we think of $Kar(\mathbf{C})$ as the Morita $n$-category of $n$-condensation monads in $\mathbf{C}$.

The notion of $n$-condensation monad reviewed above is not unital. However, in order to compare it with the definitions of other algebraic structures, it is necessary to impose some unitality conditions.
%We begin with the definition of 0-unitality for an $n$-condensation and work our way up to $n-1$-unitality.
Provided that $n\geq 2$, we say that an $n$-condensation is $0$-unital if $f$ is left adjoint to $g$. More precisely, there exists a $2$-morphism $\eta:\mathrm{id}_X\Rightarrow g\circ f$ satisfying the triangle identities with $\phi:f\circ g\Rightarrow \mathrm{id}_Y$, potentially up to higher coherences. We emphasize that this is a property and not additional structure. For our purposes, we will need other unitality conditions. For $n\geq 3$, we will say that an $n$-condensation $(X,Y,f,g,...)$ is $1$-unital if the corresponding $(n-1)$-condensation $(f\circ g, \mathrm{id}_X, ...)$ is $0$-unital. More generally, for any integer $0\leq k\leq n-1$, we define inductively the notion of a $k$-unital $n$-condensation. Then, we say that an $n$-condensation is fully unital if it is $k$-unital for all $k$'s. Mimicking the procedure by which we obtain the definition of an $n$-condensation monad from that of an $n$-condensation, we arrive at the definition of a $k$-unital, resp.\ fully unital, $n$-condensation monad. Given a fully unital $n$-condensation monad $(X,e,...)$, we will also say that $(e,...)$ is a separable algebra in the monoidal $(n-1)$-category $\mathrm{End}_{\mathbf{C}}(X)$. In other words, a separable algebra is a fully unital condensation algebra.

It was argued in \cite[Theorem 3.1.7]{GJF} that if the $n$-category $\mathbf{C}$ has adjoints for $1$-morphisms and $Hom$-$(n-1)$-categories have all condensates, every $n$-condensation monad is Morita equivalent, i.e.\ it is equivalent as an object of $Kar(\mathbf{C})$, to a $0$-unital $n$-condensation monad. In particular, if $\mathbf{C}$ has adjoints for $k$-morphisms for some $1\leq k\leq (n-1)$, it follows that every $n$-condensation monad is Morita equivalent to a $(k-1)$-unital one. Provided that $\mathbf{C}$ has adjoints for morphisms of all levels, it is natural to wonder whether every $n$-condensation monad is Morita equivalent to a fully unital one. This does not follow immediately from the construction given in \cite[Theorem 3.1.7]{GJF}. Nevertheless we will making the following assumption:
\begin{Assumption*}
    If $\mathbf{C}$ has adjoints for morphisms of all levels, then every $n$-condensation monad is Morita equivalent to a fully unital one.
\end{Assumption*}
 \noindent Said differently, we expect that every $n$-condensation algebra is Morita equivalent to a separable algebra. Below, we will freely use this claim for $n\leq 4$.

For later use, we now unpack the above notion of a separable algebra for small values of $n$. In the case $n=2$, separable algebras in the sense above correspond to unital, but non-counital, special Frobenius algebras. In the monoidal 1-category of vector spaces $\mathbf{Vect}$, there is a more familiar notion of separable algebra:\ A (unital) algebra $A$ in $\mathbf{Vect}$ is called separable if its multiplication map $m:A\otimes A\rightarrow A$ admits a section as a map of $A$-bimodules. It was shown in \cite{eilenberg1956dimension} that any separable algebra $A$ can be given the structure of a special Frobenius algebra. This motivates our choice of terminology. Moreover, observe that being separable is a property of an algebra and not additional structure. In the case $n=3$, our notion of separable algebra coincides with that given in \cite{D7}. Firstly, recall that a rigid algebra in a monoidal 2-category is a unital algebra $A$, whose multiplication map $m:A\otimes A\rightarrow A$ admits a right adjoint $m^R$ as a map of $A$-bimodules. Then, a rigid algebra is called separable provided that the counit $\epsilon$ of the adjunction $m\dashv m^R$ admits a section as a map of $A$-bimodules. Rigid algebras in the monoidal 2-category $\mathbf{2Vect}$ of finite semisimple 1-categories are precisely multifusion 1-categories. In this case, it is known \cite{ENO1, DSPS13, JF} that every rigid algebra is in fact a separable. Of most interest to us will be the case $n=4$. We begin by introducing the notion of a rigid algebra in a monoidal 3-category. A rigid algebra is a (unital) algebra $(A,m,...)$ whose multiplication map $m:A\otimes A\rightarrow A$ has a right adjoint $m^R$ as a map of $A$-bimodules, and the counit $\epsilon$ of the adjunction $m\dashv m^R$ has a right adjoint $\epsilon^R$ as a 2-morphism of $A$-bimodules. By inspection, a separable algebra in $End_{\mathbf{C}}(X)$ is a rigid algebra $(A,m,...)$ such that the counit $\Upsilon$ of the adjunction $\epsilon\dashv\epsilon^R$ has a section as a $3$-morphism of $A$-bimodules.

\subsection{Higher Fusion Categories}

We can define fusion $n$-categories (over $\mathbb{C}$) inductively as follows:\ The starting point of our inductive definition, that is, the case $n=1$, is taken to be the classical definition as in \cite{EGNO}. In the case $n=2$, our definition recovers \cite[Definition 2.1.6]{DR} (or \cite{D2} to account for our use of weak 2-categories). For $n\geq 2$, we make the following string of definitions:
\begin{itemize}
\item An $n$-category $\mathbf{C}$ is called \textbf{locally finite semisimple} if for every object $X \in \mathbf{C}$, $\Omega_X \mathbf{C} := End_{\mathbf{C}}(X)$ is a multifusion $(n-1)$-category. In particular, every $k$-morphism for $0<k<n$ has both a left and a right adjoint.

\item An $n$-category $\mathbf{C}$ is called \textbf{semisimple} if it is locally finite semisimple, and it is Cauchy complete in the sense of \cite{GJF}, that is, every $n$-condensation monad splits and direct sums for objects exist. Every locally finite semisimple $n$-category $\mathbf{C}$ embeds into the semisimple $n$-category $Cau(\mathbf{C})$ given by its Cauchy completion \cite[Theorem 4.3.5]{GJF}. Explicitly, $Cau(\mathbf{C})$ is the condensation completion of the direct sum completion of $\mathbf{C}$.

Physically, Cauchy completion is the statement of condensation completion:\ i.e.\ we include all possible $n$-dimensional topological defects that can be constructed by condensing/gauging topological defects living inside an $n$-dimensional topological defect $X$, which the mutlifusion $(n-1)$ category $\Omega_X \mathbf{C}$. Here it should be noted that we include an $n$-dimensional TFT stacked on top of a $d>n$-dimensional theory as an $n$-dimensional topological defect, or equivalently as an object in $\mathbf{C}$.

\item An object $X$ of a semisimple $n$-category $\mathbf{C}$ is called \textbf{simple} if $\Omega_X \mathbf{C}$ is fusion. 

Physically, this means that there are no genuine topological point-like defects on $X$ except for multiples of the identity local operator.
\item Two simple objects are called \textbf{connected} if there exists a non-zero 1-morphism between them. If $\mathbf{C}$ is locally finite, being connected defines an equivalence relation on the set of equivalence classes of simple objects of $\mathbf{C}$, and we write $\pi_0(\mathbf{C})$ for the corresponding quotient.

Physically, two $n$-dimensional topological defects are connected if there exists a non-zero topological interface between them. Two topological defects that are connected are also related by condensations, or in other words either of the two defects can be obtained by gauging topological sub-defects living on the other defect, possibly stacked with an $n$-dimensional TFT. Thus, $\pi_0(\mathbf{C})$ characterizes gauge inequivalent classes of topological $n$-dimensional defects.
\item A \textbf{finite semisimple $n$-category} is a semisimple $n$-category such that $\pi_0(\mathbf{C})$ is finite. For $n=1$, connected components are exactly (equivalence classes of) simple objects. For $n=2$, every connected component contains only finitely many simple objects \cite{DR}. For $n = 3$, non-empty connected components contain infinitely many simple objects. Physically, this is because there are an infinite number of 3d TFTs.
\item A \textbf{multifusion} $n$-category $\mathbf{C}$ is a monoidal finite semisimple $n$-category that is rigid, i.e.\ every object has a left and a right dual.

Physically, such an $n$-category describes topological defects that can be bent into $\cup$ and $\cap$ shapes.
\item A \textbf{fusion} $n$-category is a multifusion $n$-category whose monoidal unit is simple.
%\footnote{The definition of a multifusion $n$-category is a priori weaker than that used in \cite{JF}. In the case $n=1$, as discussed in the aforementioned reference, they are actually equivalent thanks to \cite{ENO1,DSPS13}. In the case $n=2$, the equivalence was established in \cite{D9}, and it is expected to hold in full generality.}
\end{itemize}

Physically, a fusion $n$-category $\mathbf{C}$ describes a collection of topological $n$-dimensional defects (satisfying certain finiteness properties discussed above) along with their fusion rules. Such a collection is required to be closed under gauging operations. If there is no additional braiding structure on $\mathbf{C}$, then $\mathbf{C}$ describes topological defects of an $(n+1)$-dimensional QFT. In the presence of additional braiding, i.e. further monoidal structures, then $\mathbf{C}$ describes topological defects in $d>n+1$ dimensional QFTs, where the difference between $d$ and $n+1$ captures the amount of further monoidality. In other words, fusion $n$-categories describe finite global symmetries of quantum systems of spacetime dimension $d>n$.

Another definition of multifusion $n$-category was put forward in \cite{JF}. We discuss how it compares with ours for small values of $n$. In the case $n=1$, our definition coincides with the classical notion of a multifusion $1$-category given for instance in \cite{EGNO}. It also agrees with the definition of a multifusion $1$-category used in \cite{JF} by \cite[Corollary II.8]{JF}. Namely, a multifusion 1-category in the sense of \cite{JF} is a separable algebra in $\mathbf{2Vect}$, the monoidal 2-category of finite semisimple 1-categories. On the other hand, a multifusion 1-category in our sense is a rigid algebra in $\mathbf{2Vect}$. These two notions agree by \cite{ENO1,DSPS13} given that we working over the complex numbers. In particular, they lead to the same notion of a finite semisimple 2-category, which is the one given in \cite{DR}.\footnote{To be precise, that the completion operations are the same was discussed in \cite[Section 3.1]{GJF}, see also \cite{D1}.} Our notion of multifusion 2-category then coincides with that introduced in \cite{DR}. It also corresponds to the notion of a rigid algebra in $\mathbf{3Vect}$, the monoidal 3-category of finite semisimple 2-categories. On the other hand, the multifusion 2-categories used in \cite[Definition II.9]{JF} are separable algebras in $\mathbf{3Vect}$ by \cite[Theorem 4.1.1]{GJF} (see also \cite[Second Proof of Theorem 1]{JF}).\footnote{Here we are in addition using our claim that every ($0$-unital) $3$-condensation algebra is Morita equivalent to a separable algebra.} That these two notions agree was discussed in \cite[Remark 5.2.4]{D9}. Then again, there is a well-defined notion of a finite semisimple 3-category. We expect that our notion of multifusion 3-category agrees with that used in \cite{JF}, although a proof is beyond the scope of our current investigations. One could then wonder whether the corresponding notions of multifusion 4-categories agree, and so on, and so forth.

\paragraph{Operations relating higher categories of different levels.}
There are two operations, that we label by $\Omega$ and $\Sigma$ that relate higher-categories differing by one level. We review these operations below.

\paragraph{Loops.}
The first operation $\Omega$ inputs a fusion $n$-category $\mathbf{C}$ and outputs the braided fusion $(n-1)$-category formed by endomorphisms of the unit object $\mathbf{1}$
\be
\Omega\mathbf{C}:=\Omega_{\mathbf{1}}(\mathbf{C})=End_{\mathbf{C}}(\mathbf{1}) \,.
\ee
More generally, $\Omega\mathbf{C}$ carries one additional level of monoidality than $\mathbb{C}$. Physically, one should think of $\Omega\mathbf{C}$ as the category formed by genuine topological $(n-1)$-dimensional defects (and their sub-defects) contained within the collection $\mathbf{C}$. Note that such topological $(n-1)$-dimensional defects can braid with each other and hence $\Omega\mathbf{C}$ is a braided fusion $(n-1)$-category.

\paragraph{Suspension.}
The other operation $\Sigma$ is more involved, and we follow \cite{GJF}. Given any multifusion $(n-1)$-category $\mathbf{D}$, we let $\mathrm{B}\mathbf{D}$ be the locally semisimple $n$-category with a single object $*$ and $End_{\mathrm{B}\mathbf{D}}(*)=\mathbf{D}$. Then, we define $$\Sigma\mathbf{D}:= Cau(\mathrm{B}\mathbf{D})$$ as the finite semisimple $n$-category obtained by taking the Cauchy completion of $\mathrm{B}\mathbf{D}$. Note that by definition, we have
\be
\Omega(\Sigma\mathbf{D})\simeq\mathbf{D} \,.
\ee
Provided that the fusion $n$-category $\mathbf{C}$ is connected, i.e.\ that $\pi_0(\mathbf{C})$ is a singleton, we also have
\be\label{e1}
\Sigma(\Omega\mathbf{C})\simeq\mathbf{C}\,.
\ee

For small values of $n$, it is known that the finite semisimple $n$-category $\Sigma\mathbf{D}$ coincides with the $n$-category $\mathbf{Mod}(\mathbf{D})$ of finite semisimple module $(n-1)$-categories over $\mathbf{D}$. Namely, for $n=2$, we have that $\mathbf{D}$ is a multifusion 1-category, and it was shown in \cite[Theorem 3.3.3]{GJF} that $\Sigma\mathbf{D}$ can be identified with the genuine Morita 2-category of separable algebras in $\mathbf{D}$. It then follows from \cite{DR,D1} that $\Sigma\mathbf{D}$ agrees with the 2-category of finite semisimple $\mathbf{D}$-module categories. One key step is Ostrik's theorem \cite{O1}, which allows to express every finite semisimple $\mathbf{D}$-module 1-category as the category of modules over an algebra in $\mathbb{D}$. It is then a consequence of the main theorem of \cite{ENO1} that this algebra is separable. In the case $n=3$, it follows from \cite[Theorem 3.3.3]{GJF} and the claim we have made above that every $3$-condensation algebra is Morita equivalent to a separable algebra. We therefore have that $\Sigma\mathbf{D}$ is equivalent to the genuine Morita $3$-category of separable algebras in $\mathbf{D}$. Appealing to \cite{D8,D9}, one finds that this Morita 3-category is equivalent to the 3-category of finite semisimple $\mathbf{D}$-module 2-categories. Again, one key step consists in proving that every finite semisimple $\mathbf{D}$-module 2-category is the 2-category of modules over a rigid algebra in $\mathbf{D}$, which was achieved in \cite{D4}. Then, one shows that this rigid algebra must be separable \cite{D9}. Now, given a multifusion 3-category $\mathbf{D}$, the above discussion motivates the equivalence of 4-categories $$\Sigma\mathbf{D}\simeq\mathbf{Mod}(\mathbf{D}),$$ that will be liberally employed below. Establishing this result rigorously would not only entail proving that every finite semisimple $\mathbf{D}$-module 3-category is equivalent to the 3-category of modules over a separable algebra in $\mathbf{D}$, but also establishing that every $4$-condensation algebra is Morita equivalent to a separable algebra.

\subsection{Fusion 3-Categories}

\paragraph{Without Braiding.}
Our focus of study are the fusion 3-categories $\cVect(A[1])$ of $A[1]$-graded 3-vector spaces for some finite abelian group $A$, where recall that $A[p]$ denotes that we are considering a $p$-form group $A$. These fusion 3-categories can be constructed as
\be\label{e2}
\cVect(A[1])=\Sigma\mathbf{2Vect}(A[0]) \,,
\ee
where $\mathbf{2Vect}(A[0])$ is the (braided) fusion 2-category of $A[0]$-graded 2-vector spaces \cite{DR}. In fact, the fusion 3-category $\cVect(A[1])$ is connected (as the higher group $A[1]$ is connected), so \eqref{e2} can be obtained by applying \eqref{e1} using that $\Omega\cVect(A[1])=\mathbf{2Vect}(A[0])$. We also think of
$$\cVect(A[1])\simeq \mathbf{Mod}(\mathbf{2Vect}(A[0]))$$
as the 3-category of finite semisimple $\mathbf{2Vect}(A[0])$-module 2-categories, or equivalently \cite{D8,D9}, as the Morita 3-category of rigid algebras in $\mathbf{2Vect}(A[0])$. This last model is very explicit. Namely, rigid algebras in $\mathbf{2Vect}(A[0])$ are simply $A$-graded multifusion 1-categories, for which the grading does not have to be faithful. We find that $\mathbf{Mod}(\mathbf{2Vect}(A[0]))$ is the Morita 3-category of $A$-graded multifusion 1-categories and $A$-graded finite semisimple bimodule 1-categories. Simple objects correspond to $A$-graded fusion 1-categories. In particular, the monoidal unit is $\mathbf{Vect}$ for which only the grade given by $0\in A$ is non-empty.

Physically, $\cVect(A[1])$ describes $(d-3)$-form symmetry with no 't Hooft anomaly in spacetime dimension $d\ge4$. The $(d-3)$-form symmetry is generated by invertible surface defects whose fusion is described by the abelian group $A$. The 3-category $\cVect(A[1])$ also contains the possible 3-dimensional topological defects that can be obtained by condensing the aforementioned surface defects. For $d=4$, the 3-category $\cVect(A[1])$ is simply a fusion 3-category, while for $d>4$, we have that $\cVect(A[1])$ is a fusion 3-category with higher levels of monoidality. For example, for $d=5$, $\cVect(A[1])$ is a braided fusion 3-category. In this paper, the case $d=4$ will play the most prominent role.

\paragraph{With Braiding.}
We will also consider connected braided fusion 3-categories other than $\cVect(A[1])$. Such categories are all of the form $\Mod( \mathfrak{S})\simeq\Sigma\mathfrak{S}$, where $\mathfrak{S}$ is a \textbf{sylleptic} fusion 2-category, that is, has three levels of monoidality. As already remarked earlier, physically speaking, we can consider a 2-category whose objects are given by surfaces in $d$ spacetime dimensions. These surfaces then have $d-2$ ambient dimensions in which to braid, and therefore form a $(d-2)$-monoidal 2-category. When $d=5$, this level of monoidality is called sylleptic \cite{BD}. In particular, a sylleptic monoidal 2-category is a braided monoidal 2-category equipped with additional structure. Recall that a braided structure on a monoidal 2-category $\mathfrak{C}$ consists of a braiding $b$, which comprises of natural equivalences
\be
b_{X|Y}: X\otimes Y\rightarrow Y \otimes X
\ee
for all objects $X,Y\in \mathfrak{C}$, together with modifications $R$ and $S$, called hexagonators, that provide the higher coherences for the monoidality of $b$, and satisfy various axioms. A syllepsis $s$ on a braided monoidal 2-category is a modification given by isomorphisms
\be
s_{X|Y}: \quad b_{X|Y} \Rightarrow b^{-1}_{Y|X}
\ee
for all objects $X,Y$ satisfying coherence conditions. The coherence data for the hexagonators and syllepsis can be found in \cite{RS} (see also \cite{SP} for an easily accessible modern account), but we will not need that level of detail in this work. 

\paragraph{The Braided Fusion 3-category $\cZ(\cVect(A[1]))$.}
The most relevant braided fusion 3-category for our work is $\cZ(\cVect(A[1]))$, the Drinfeld center of $\cVect(A[1])$. This 3-category describes the topological operators in Dijkgraaf-Witten theory for $A[1]$ in $(4+1)$d, i.e.\ the SymTFT for a $(3+1)$d theory with 1-form symmetry $A$. This theory has an action given by 
\be
S=\int b_2\cup\delta c_2\,,
\ee
where $b_2$ is a 2-cochain valued in the abelian group $A$ and $c_2$ is a 2-cochain valued in the Pontryagin dual abelian group $\widehat A$.
It follows from the results of \cite{JF} that 
\be
\Omega\cZ\big(\cVect(A[1])\big)=\mathbf{2Vect}^{\varsigma}\big(A[0]\oplus \widehat{A}[0]\big)
\ee
as sylleptic fusion 2-categories, where the braiding $b$ is trivial, and $\varsigma$ is a canonical syllepsis that will be introduced below. Moreover, it follows from an $S$-matrix argument \cite{JF4}, also known as remote detectability, that $\cZ(\cVect(A[1]))$ is connected, so we have an equivalence of braided fusion 3-categories
\be\label{Grumpy}
\cZ\big(\cVect(A[1])\big)=\Sigma\bVect^\varsigma\big(A[0]\oplus \widehat{A}[0]\big) \,.
\ee
The fusion 2-category $\bVect^\varsigma\big(A[0]\oplus \widehat{A}[0]\big)$ captures topological Wilson surfaces constructed from the gauge fields $b_2$ and $c_2$, with $\varsigma$ capturing the non-trivial linking between these two types of surfaces. Since there are no fundamental fields that can give rise to 3-dimensional Wilson topological defects, all such defects have to be condensations of the above Wilson surfaces. Thus, topological defects of codimension-2 and above of the SymTFT $\mathfrak{Z}(\cVect(A[1]))$ are precisely captured by the 3-category $\Sigma\bVect^\varsigma\big(A[0]\oplus \widehat{A}[0]\big)$, which is the content of the above equation.

%%%%%%%%%%%%%%%%%%%%%%%%%%%%%%%%%%%%%%%%%%%%%%%%%%%%%
\subsection{Sylleptic strongly fusion 2-categories}\label{subsection:sylleptic}
%%%%%%%%%%%%%%%%%%%%%%%%%%%%%%%%%%%%%%%%%%%%%%%%%%%%%
As discussed above in (\ref{Grumpy}), our primary object of study $\cZ(\cVect(A[1])$ is the condensation completion of a sylleptic fusion 2-category $\mathbf{2Vect}^\varsigma((A\oplus\widehat{A})[0])$. In fact, $\mathbf{2Vect}^\varsigma((A\oplus\widehat{A})[0])$ is a \textbf{strongly} fusion 2-category, or in other words, there are no genuine topological line operators (other than the identity) in this 2-category. Motivated by this, we review the general theory of (sylleptic) strongly fusion 2-categories.

\paragraph{Without Braiding.}
A fusion 2-category $\mathbf{C}$ has no line operators if it satisfies
\be
\Omega\mathbf{C} \simeq \mathbf{Vect}\,.
\ee
Such fusion 2-categories are referred to as \textbf{strongly fusion 2-categories}.\footnote{Sometimes the definition of strongly fusion 2-categories is taken to be $\Omega\mathbf{C} \simeq \mathbf{Vect}$ or $\Omega\mathbf{C} \simeq \mathbf{sVect}$, where in the latter case we do have a non-trivial line operator. The former ones are then referred to as bosonic strongly fusion 2-categories, while the latter ones are referred to as fermionic strongly fusion 2-categories. As the latter notion will not appear in this work, we have decide to keep our nomenclature brief.}  Furthermore, it is known that if $\mathbf{C}$ is a strongly fusion 2-category, then the equivalence classes of simple objects of $\mathbf{C}$ form a group $G$ under the tensor product \cite{JFY1}. In particular, it follows that every strongly fusion 2-category is of the form $\mathbf{2Vect}^{\pi}(G[0])$ for some class $\pi$ in $H^4(G[1];\mathbb{C}^{\times})$. 

\paragraph{With Braiding but Without Syllepsis.}
We will focus on the case when the group $G=A$ is abelian. This is manifestly a necessary condition for the corresponding fusion 2-category to admit a braiding, and therefore also a syllepsis. It is a well-known homotopy-theoretic fact that all braided monoidal structures on $\mathbf{2Vect}(A[0])$ are parameterized by the group $H^5(A[2];\mathbb{C}^{\times})$.\footnote{We use the notation $H^p(A[q];\mathbb{C}^{\times})=H^p(B^{q}A;\mathbb{C}^{\times})$. In particular, the standard group cohomology is denoted as $H^p(A[1];\mathbb{C}^{\times})$.} We note that the underlying monoidal structure may not be trivial

Physically, a braiding on $\mathbf{2Vect}(A[0])$ can be regarded as a 't Hooft anomaly of an $A$ 1-form symmetry in 4d, which explains why they are parametrized by the group $H^5(A[2];\mathbb{C}^{\times})$.

\paragraph{With Syllepsis.}
Given a finite abelian group $A$, we wish to describe the possible sylleptic structures on $\mathbf{2Vect}(A[0])$. These consist of a choice of braiding together with a compatible syllepsis. It follows from a standard homotopy theoretic argument that sylleptic structures on $\mathbf{2Vect}(A[0])$ are parameterized by $H^6(A[3];\mathbb{C}^{\times})$.

Physically, a sylleptic structure on $\mathbf{2Vect}(A[0])$ can be regarded as a 't Hooft anomaly of an $A$ 2-form symmetry in 5d, which explains why they are parametrized by the group $H^6(A[3];\mathbb{C}^{\times})$.

As explained for instance in \cite{JFY2}, the group $H^6(A[3];\mathbb{C}^{\times})$ admits an explicit description. Let us write $A_2$ for the subgroups of elements of order $2$, and $\widehat{\Lambda^2 A}$ for the group of alternating 2-forms on $A$, that is additive maps $\lambda:A\otimes_{\mathbb{Z}} A\rightarrow \mathbb{C}^{\times}$ such that $\lambda(a,a)=1$ for every $a\in A$.

\begin{Lemma}
Let $A$ be a finite abelian group. There is an isomorphism of groups
\be
H^6(A[3];\mathbb{C}^{\times})\cong \widehat{A_2}\oplus \widehat{\Lambda^2 A}\,.
\ee
In particular, sylleptic structures on the fusion 2-category $\mathbf{2Vect}(A[0])$ (up to grading preserving equivalences) are classified by $\widehat{A_2}\oplus \widehat{\Lambda^2 A}$.
\end{Lemma}

\noindent The summand $\widehat{A_2}$ parameterizes the allowed braidings, and the summand $\widehat{\Lambda^2 A}$ parameterizes the allowed syllepsis. In particular, $\widehat{A_2}$ is a subgroup of the group parameterizing all possible braidings on $\mathbf{2Vect}(A[0])$, namely the subgroup of all possible braidings that are compatible with the existence of a sylleptic structure. Furthermore, given a syllepsis $s$ on $\mathbf{2Vect}(A[0])$, the corresponding alternating 2-form is given by 
\be
\mathsf{Alt}(s)(a,b):=s(a,b)/s(b,a)
\ee
for every $a,b\in A$.

It follows from the above considerations that every sylleptic strongly fusion 2-category is of the form 
\be
\mathbf{2Vect}^{(b,s)}(A[0])
\ee
for a finite abelian group $A$ and a class $(b,s)$ in $H^6(A[3];\mathbb{C}^{\times})$, where $b$ is a braiding and $s$ is a compatible syllepsis. 

In our applications, we will take $b$ to be trivial as this holds in the case of $\Omega\cZ(\cVect(A[1]))$ (because of the existence of a fiber 3-functor). Also note that this is automatic if $A$ has odd order.
Moreover, we will use the canonical alternating 2-form $\mathsf{Alt}(\varsigma)$ on $A\oplus\widehat{A}$ defined by \be
\mathsf{Alt}(\varsigma)(a_1,\lambda_1,a_2,\lambda_2):=\lambda_1(a_2)/\lambda_2(a_1)\,,
\ee
and we write $\varsigma$ for the corresponding class in cohomology. This yields a sylleptic structure on the fusion 2-category $\mathbf{2Vect}^\varsigma((A\oplus\widehat{A})[0])$ for which \be\label{eq:centerequiv}
\cZ(\cVect(A[1])\simeq \Sigma\mathbf{2Vect}^\varsigma((A\oplus\widehat{A})[0])
\ee
is an equivalence of braided fusion 3-categories. 

%%%%%%%%%%%%%%%%%%%%%%%%%%%%%%%%%%%%%%%%%%%%%%%%%%%%%
\subsection{Higher Morita Categories:\ $\Pic$ and $\mathscr{W}itt$}
\label{sec:HigherMorita}
%%%%%%%%%%%%%%%%%%%%%%%%%%%%%%%%%%%%%%%%%%%%%%%%%%%%%
We now consider sylleptic (but not necessarily strongly) fusion 2-categories. Recall that physically these describe topological surface defects in 5d. Let $\mathfrak{S}$ be a sylleptic fusion 2-category. We are interested in understanding the fusion 4-category $\Sigma^2\mathfrak{S}$. Physically, this describes 4-dimensional topological defects (and their sub-defects) in the 5d theory that can be obtained by condensing the defects in $\mathfrak{S}$. In the particular case of $\mathfrak{S}=\mathbf{2Vect}^\varsigma((A\oplus\widehat{A})[0])$, the fusion 4-category $\Sigma^2\mathfrak{S}$ describes topological defects of the SymTFT $\mathfrak{Z}(\cVect(A[1]))$. 

We will be especially interested in the collection of invertible 4-dimensional topological defects. In the particular case of $\mathfrak{S}=\mathbf{2Vect}^\varsigma((A\oplus\widehat{A})[0])$, this space is intimately related to the $G$-graded extensions of $\cVect(A[1])$. For $G=\mathbb{Z}/2,\mathbb{Z}/4$, such extensions construct the fusion 3-categories corresponding to the familiar duality defects. In fact, in \cite{Kaidi:2022cpf}, the SymTFT for these duality symmetries $G$ are constructed by gauging $G$ 0-form symmetries (comprising of condensation defects) of $\mathfrak{Z}(\cVect(A[1])$.

\paragraph{The Picard Morita 4-Category.}
We begin by considering a 4-category that can be associated to any fusion 3-category $\mathscr{C}$. This is the Morita 4-category of rigid algebras in $\mathscr{C}$, which we will identify with the 4-category $\mathbf{Mod}(\mathscr{C})$ formed by finite semisimple module 3-categories over $\mathscr{C}$.

In the case that the fusion 3-category is the condensation completion of a sylleptic fusion 2-category $\mathfrak{S}$, that is
\be
\mathscr{C}=\Sigma\mathfrak{S}
\ee
the 4-category $\mathbf{Mod}(\Sigma\mathfrak{S})$ is a fusion 4-category. Upon unpacking, we find that rigid algebras in $\Sigma\mathfrak{S}$ are exactly $\mathfrak{S}$-central multifusion 2-categories, which can be identified as the objects of $\mathbf{Mod}(\Sigma\mathfrak{S})$.

We are especially interested in the invertible objects of this 4-category.
\begin{Definition}\label{def:Picardspace}
Let $\mathfrak{S}$ be a sylleptic fusion 2-category. We write $\mathscr{P}ic(\Sigma\mathfrak{S})$ for the Picard space $(\Sigma^2\mathfrak{S})^{\times}$ of the fusion 4-category $\Sigma^2\mathfrak{S}$, that is, the space consisting of invertible objects and invertible morphisms. We also set ${Pic}(\Sigma\mathfrak{S}):=\pi_0(\mathscr{P}ic\big(\Sigma\mathfrak{S})\big)$, the group of equivalence classes of invertible objects of $\Sigma^2\mathfrak{S}$.
\end{Definition}
% Re-expressing $\Sigma^2\mathfrak{S}$ as
% \be
% \Sigma^2\mathfrak{S}\simeq \mathbf{Mod}(\Sigma\mathfrak{S})
% \ee
% we can regard it as the  

% which characterizes invertible 0-form symmetries of a 5d theory carrying $\mathfrak{S}$ surface defects.

% \begin{Definition}\label{def:Picardspace}
% Let $\mathfrak{S}$ be a sylleptic fusion 2-category. We write $\mathscr{P}ic(\Sigma\mathfrak{S})$ for the Picard space $(\Sigma^2\mathfrak{S})^{\times}$ of the fusion 4-category $\Sigma^2\mathfrak{S}$. We also set ${Pic}(\Sigma\mathfrak{S}):=\pi_0(\mathscr{P}ic\big(\Sigma\mathfrak{S})\big)$, the group of equivalence classes of invertible objects of $\Sigma^2\mathfrak{S}$.
% \end{Definition}

Physically, ${Pic}(\Sigma\mathfrak{S})$ is the group formed by invertible topological 4-dimensional defects in $\Sigma^2\mathfrak{S}$, which describes a 0-form symmetry that must be carried by a 5d theory containing topological surfaces described by $\mathfrak{S}$. In particular, these are invertible 4-dimensional condensation defects arising from surfaces in $\mathfrak{S}$. 

% For example for $\mathfrak{S}=\mathbf{2Vect}(A[0])$, namely for a 2-form symmetry described by an abelian group $A$ in 5d, we can construct a condensation surface by (higher-)gauging\footnote{We gauge the full $A$ symmetry faithfully, and do not include gaugings of any additional non-faithfully acting symmetries.} the surface defects implementing this two form symmetry along a codimension-1 submanifold of the 5d spacetime. There are a total of $H^4(A[2],U(1))$ worth of such invertible condensation defects. In particular, for $A=\mathbb{Z}_2$, we have $H^4(A[2],U(1))=\Z_4$, and hence
% \be
% \Z_4 \subseteq{Pic}(\Sigma\mathbf{2Vect}(A[0])) \,.
% \ee

\begin{Remark}
As $\mathfrak{S}$ is a sylleptic fusion 2-category, the space $\mathscr{P}ic(\Sigma\mathfrak{S})$ inherits the structure of a grouplike topological monoid, i.e.\ it comes equipped with a group structure. In particular, it makes sense to consider its delooping $\mathrm{B}\mathscr{P}ic(\Sigma\mathfrak{S})$.
\end{Remark}

\paragraph{The Witt Morita 4-Category.}
On the other hand, we can consider an apriori different Morita 4-category $\mathsf{Mor}_2^{sep}(\mathfrak{S})$ of braided rigid algebras in $\mathfrak{S}$. Indeed, it follows from a slight generalization of \cite{JFS}, that we can consider the Morita 4-category of braided algebras in any braided monoidal 2-category with suitable colimits, and the existence of such colimits for braided rigid algebras was established in \cite{D8,D9}. Moreover, as $\mathfrak{S}$ is a sylleptic fusion 2-category, $\mathsf{Mor}_2^{sep}(\mathfrak{S})$ inherits a monoidal structure.\footnote{The underlying 4-category of $\mathsf{Mor}_2^{sep}(\mathfrak{S})$ only depends on the underlying braided monoidal structure of $\mathfrak{S}$.} We will specifically be interested in the invertible objects of this monoidal 4-category.

\begin{Definition}
Let $\mathfrak{S}$ be a sylleptic fusion 2-category. We write $\mathscr{W}itt(\mathfrak{S})$ for the Picard space of the monoidal 4-category $\mathsf{Mor}_2^{sep}(\mathfrak{S})$. We also set $\mathcal{W}itt(\mathfrak{S}):=\pi_0(\mathscr{W}itt(\mathfrak{S}))$, the Witt group of invertible braided rigid algebras in $\mathfrak{S}$.
\end{Definition}

\paragraph{Equivalence of Pic and Witt Morita 4-Categories.}
We assert that the above two Morita 4-categories are actually equivalent.
\begin{Theorem}
There is an equivalence of monoidal 4-categories between the Morita 4-category $\mathsf{Mor}_2^{sep}(\mathfrak{S})$ of rigid braided algebras in $\mathfrak{S}$ and the Morita 4-category $\mathbf{Mod}(\Sigma\mathfrak{S})$ of $\mathfrak{S}$-central multifusion 2-categories. The equivalence is given by sending a braided rigid algebra $\mathcal{B}$ in $\mathfrak{S}$ to the corresponding $\mathfrak{S}$-central braided multifusion 2-category $\mathbf{Mod}_{\mathfrak{S}}(\mathcal{B})$ of $\mathcal{B}$-modules in $\mathfrak{S}$.
\end{Theorem}

\begin{proof}
Let $\mathcal{B}$ be a braided rigid algebra in $\mathfrak{S}$. It was shown in \cite{DY22} that $\mathbf{Mod}_{\mathfrak{S}}(\mathcal{B})$ is a multifusion 2-category. In fact, a slight elaboration of this argument shows that $\mathbf{Mod}_{\mathfrak{S}}(\mathcal{B})$ inherits a canonical $\mathfrak{S}$-central structure. Moreover, using the 3-universal property of the relative 2-Deligne tensor product \cite{D10}, one checks at once that the assignment $\mathbf{Mod}_{\mathfrak{S}}(-)$ is monoidal (see also \cite{JF} for related results in the symmetric case).

We begin by showing that every $\mathfrak{S}$-central multifusion 2-category can be obtained in this way, that is, that the functor $\mathbf{Mod}_{\mathfrak{S}}(-)$ is essentially surjective. In fact, as every multifusion 2-category is Morita equivalent to a direct sum of fusion 2-categories, it is  enough to prove that every $\mathfrak{S}$-central fusion 2-category is in the image. To this end, let us fix an arbitrary $\mathfrak{S}$-central fusion 2-category $\mathfrak{C}$, that is, a fusion 2-category $\mathfrak{C}$ equipped with a braided functor $F:\mathfrak{S}\rightarrow\mathcal{Z}(\mathfrak{C})$. It follows from \cite{D9} that $\mathfrak{C}$ is Morita equivalent to a connected fusion 2-category $\mathbf{Mod}(\mathcal{A})$, with $\mathcal{A}$ a braided fusion 1-category. As a Morita equivalence between fusion 2-categories induces a braided equivalence at the level of Drinfeld centers, we may as well assume that $\mathfrak{C}=\mathbf{Mod}(\mathcal{A})$. Now, we may view $\mathbf{Mod}(\mathcal{A})$ as a finite semisimple $\mathfrak{S}$-module 2-category. Thanks to \cite{D4,D9}, we have that there exists a rigid algebra $\mathcal{B}$ in $\mathfrak{S}$ such that $\mathbf{Mod}(\mathcal{A})\simeq \mathbf{Mod}_{\mathfrak{S}}(\mathcal{B})$ as finite semisimple 2-categories. We claim that $\mathcal{B}$ admits a braided structure so that the equivalence $\mathbf{Mod}(\mathcal{A})\simeq \mathbf{Mod}_{\mathfrak{S}}(\mathcal{B})$ is compatible with the monoidal structures. Namely, $\mathcal{B}$ as an algebra is equivalent to $F^*\circ U^*(I)$, the image of the monoidal unit of $\mathbf{Mod}(\mathcal{A})$ under the right adjoint of the canonical monoidal functor $U\circ F:\mathfrak{S}\rightarrow \mathbf{Mod}(\mathcal{A})$ given by composing $F$ with the forgetful functor $U:\mathcal{Z}(\mathfrak{C})\rightarrow\mathfrak{C}$. The fact that this functor factors through the Drinfeld center endows $\mathcal{B}$ with the desired braiding. (This is a categorification of lemma 3.5 of \cite{DMNO}.)
%%%%%%%(More precisely, there is a monoidal equivalence $\mathbf{Mod}(\mathcal{A})\simeq \mathbf{Mod}_{\mathscr{Z}(\mathfrak{C})}(U^*(I))$. Then, as $F$ is sylleptic monoidal, $F^*$ is lax sylleptic monoidal, so that $F^*\circ U^*(I)$ inherits a canonical braided structure. Finally, it follows from the fact that 2-condensations are preserved by all 2-functors that the canonical 2-functor $\mathbf{Mod}_{\mathscr{Z}(\mathfrak{C})}(U^*(I))\rightarrow \mathbf{Mod}_{\mathfrak{S}}(F^*\circ U^*(I))$ induced by $F^*$ is monoidal. [It is an equivalence by construction.])
This concludes the proof of the fact that $\mathbf{Mod}_{\mathfrak{S}}(-)$ is essentially surjective.

As $\mathfrak{S}$ is sylleptic, that $\mathbf{Mod}_{\mathfrak{S}}(-)$ is fully faithful follows from the fact that for any multifusion 2-category $\mathfrak{C}$ the 3-functor $\mathbf{Mod}_{\mathfrak{C}}(-)$ induces an equivalence between the 3-category of finite semisimple $\mathfrak{C}$-module categories and the Morita 3-category of rigid algebras in $\mathfrak{C}$ as was shown in \cite{D4,D8}.
%%%%Use the syllepsis to go from bimodules to modules, and then this becomes easy.
\end{proof}

\noindent Physically, a braided rigid algebra in $\mathfrak{S}$ describes a 3-dimensional topological boundary of 4-dimensional condensation defects arising from $\mathfrak{S}$. The Morita equivalence on such algebras relates the different boundaries of the same condensation 4-dimensional defect. Thus, the Morita category $\mathsf{Mor}_2^{sep}(\mathfrak{S})$ describes the condensation 4-dimensional defects themselves, by first describing their boundaries and then identifying the different boundaries. This is the physical content of the equivalence
\be
\mathbf{Mod}(\Sigma\mathfrak{S})\simeq \mathsf{Mor}_2^{sep}(\mathfrak{S})
\ee

The above theorem straightforwardly leads to the following corollary:

\begin{Corollary}
The spaces $\mathscr{P}ic(\mathbf{Mod}(\mathfrak{S}))$ and $\mathscr{W}itt(\mathfrak{S})$ are equivalent as grouplike topological monoids. In particular, we have a group isomorphism
\be
Pic(\Sigma\mathfrak{S})\cong \mathcal{W}itt(\mathfrak{S}) \,.
\ee
\end{Corollary}

\begin{Example}
Let us unpack the above discussion in the case $\mathfrak{S} = \bVect$. This just describes the identity topological surface, which exists in any 5d theory. 

Then, $\mathsf{Mor}^{sep}_2(\bVect)$ is the symmetric monoidal 4-category of braided multifusion 1-categories considered in \cite{BJS}. Physically, $\mathsf{Mor}^{sep}_2(\bVect)$ describes 4-dimensional topological defects that must exist in every 5d theory. These are precisely 4d TFTs, which can be converted into defects in any 5d theory by simply inserting the 4d TFT along a 4d submanifold of the spacetime of the 5d theory.\footnote{The 4d and 5d systems remain decoupled.} The braided multifusion 1-categories are formed by topological line operators arising on the 3-dimensional topological boundaries of these 4-dimensional defects. Morita equivalence relates braided multifusion 1-categories that arise on different boundaries of the same 4-dimensional defect. Physically, Morita equivalence is the statement that the different boundaries are related by gauging operations.

Furthermore, recall from \cite{BJSS} that invertible objects in the Morita 4-category $\mathsf{Mor}^{sep}_2(\bVect)$ of braided multifusion 1-categories correspond exactly to non-degenerate braided fusion 1-categories. Indeed, the 4d TFT obtained by performing Crane-Yetter construction starting from a non-degenerate braided fusion 1-category is invertible.

Two non-degenerate braided fusion 1-categories are equivalent in $\mathsf{Mor}^{sep}_2(\bVect)$ if and only if they are Witt equivalent in the sense of \cite{DMNO}. Physically speaking, they are related by gauging operations, which equivalently can be phrased in terms of the existence of a topological interface between the topological boundaries carrying two non-degenerate braided fusion 1-categories. We therefore find that $\mathcal{W}itt(\mathbf{2Vect})$ is precisely the Witt group of non-degenerate braided fusion 1-categories introduced in \cite{DMNO}.

Slightly more generally, one may fix $\mathcal{E}$ a symmetric fusion 1-category, and take $\mathfrak{S} = \mathbf{Mod}(\mathcal{E})$. Physically, $\mathcal{E}$ describes topological line operators of some 5d theory, and $\mathfrak{S} = \mathbf{Mod}(\mathcal{E})$ describes the topological surface defects that can be obtained by condensing lines in $\mathcal{E}$.
Then, $\mathsf{Mor}^{sep}_2(\mathbf{Mod}(\mathcal{E}))$ is the symmetric monoidal 4-category of braided multifusion 1-categories equipped with a central symmetric functor from $\mathcal{E}$. Again the aforementioned braided multifusion 1-categories describe topological lines on topological boundaries of 4-dimensional defects that can arise by condensing surfaces in $\mathbf{Mod}(\mathcal{E})$, or equivalently, lines in $\mathcal{E}$. In this case, it is straightforward to check using the results of \cite{BJSS} that $\mathcal{W}itt(\mathbf{Mod}(\mathcal{E}))$ is precisely the generalized, or relative, Witt group from \cite{DNO}.
\end{Example}

It follows from the previous example that the invertible objects of $\mathsf{Mor}_2^{sep}(\mathfrak{S})$ may be thought of as generalizations of non-degenerate braided fusion 1-categories. In particular, when $\mathfrak{S}$ is a sylleptic strongly fusion 2-category, i.e.\ $\mathfrak{S}\cong\mathbf{2Vect}^s(A[0])$, we will see in section \ref{section:crossedbraided} that these invertible objects are given by $A$-graded braided fusion 1-categories satisfying a generalized non-degeneracy condition depending on $s$. In fact, when $s$ is trivial, the condition is merely that the braided fusion 1-category is non-degenerate in the classical sense.

% We want to apply this theorem to the invertible objects in these 4-categories. For this purpose, we first make a couple of definitions:

%  We may now make the following definition.

%%%%%%The fact that the invertible objects of $\mathsf{Mor}_2(\bVect)$ are identified with $Pic(\Sigma \mathfrak{S})$ means that when we consider the (3+1)d boundary manipulations from the point of view of the SymTFT, the 0-form symmetry operators of the SymTFT can end on elements of $\Pic(\Sigma \mathfrak{S})$. \thib{Literally no clue what this is supposed to mean. Also there are a bunch of typos.}
% We derive some general properties of the group $\mathcal{W}itt(\mathfrak{S})$ in section \ref{subsubsection:highercatmethods}, which can be applied to the case $\mathfrak{S}\cong\mathbf{2Vect}^s(A[0])$ to gather some insight into the structure of the associated generalized Witt group.

%%%%%%%%%%%%%%%%%%%%%%%%%%%%%%%%%%%%%%%%%%%
\section{Extension Theory for Higher Fusion Categories}\label{section:extensiontheory}
%%%%%%%%%%%%%%%%%%%%%%%%%%%%%%%%%%%%%%%%%%%

%%%%%%%%%%%%%%%%%%%%%%%%%%%%%%%%%%%%%%%%%%%
\subsection{Brauer-Picard Spaces}
\label{sec:BPS}
%%%%%%%%%%%%%%%%%%%%%%%%%%%%%%%%%%%%%%%%%%%

\begin{figure}
    \centering
       \begin{tikzpicture}
 \begin{scope}[shift={(0,0)}] 
\draw [cyan, fill=cyan, opacity =0.2]
(0,0) -- (0,4) -- (2,5) -- (5,5) -- (5,1) -- (3,0)--(0,0);
\draw [black, thick, fill=white,opacity=1]
(0,0) -- (0, 4) -- (2, 5) -- (2,1) -- (0,0);
\draw [cyan, thick, fill=cyan, opacity=0.8]
(0,0) -- (0, 4) -- (2, 5) -- (2,1) -- (0,0);
% \draw [blue, thick, fill=blue, opacity=1]
% (0.5,2.25) -- (0.5, 4.25) -- (1.7, 4.85) -- (1.7,2.85) -- (0.5,2.25);
\draw [blue, thick, fill=blue, opacity=0.4]
(0.5,2.25) -- (3.5, 2.25) -- (4.7, 2.85) -- (1.7, 2.85) -- (0.5,2.25);
\node at (1.2,5) {$\mathfrak{B}^{\text{sym}}_{\mathscr{C}}$};
\draw[dashed] (0,0) -- (3,0);
\draw[dashed] (0,4) -- (3,4);
\draw[dashed] (2,5) -- (5,5);
\draw[dashed] (2,1) -- (5,1);
\draw [blue,line width = 0.05cm] (0.5, 2.25) -- (1.7, 2.85) ;
\draw [black,line width = 0.05cm] (0.5, 1.25) -- (1.7, 1.85) ;
\draw [black,line width = 0.05cm] (0.5, 3.25) -- (1.7, 3.85) ;
\draw [black]
(3,0) -- (3, 4) -- (5, 5) -- (5,1) -- (3,0);
\node[black, left] at (-1.3, 3.4) {$D_3\in\mathscr{C}$};
\draw[->] (-1.3,3.4) to [out=30,in=150]  (0.7,3.4);
\node[blue] at (2.6, 2.5) {$D_4$};
\draw[->] (-1.3,2.4) to [out=30,in=150]  (0.7,2.4);
\node[blue, left] at (-1.3, 2.4) {$D_3 \in\mathbf{B}(D_4)\in \text{\bf Bimod}(\mathscr{C})$};
\node[black, left] at (-1.3, 1.4) {$D_3\in\mathscr{C}$};
\draw[->] (-1.3,1.4) to [out=30,in=150]  (0.7,1.4);
\node at (2.6, 4.2) {$\mathfrak{Z}(\mathscr{C})$};
\end{scope}
\end{tikzpicture}
    \caption{SymTFT $\mathfrak{Z}(\mathscr{C})$ for the symmetry $\mathscr{C}$, which is realized on the  symmetry boundary $\mathfrak{B}^{\text{sym}}_{\mathscr{C}}$. A 4-dimensional topological defect $D_4$ of the SymTFT is characterized by a collection of topological defects $D_3$ that can arise at an end of $D_4$ along the symmetry boundary. These ends form a 3-category $\mathbf{B}(D_4)$ that is acted upon by genuine topological defects living in fusion 3-category $\mathscr{C}$ on the symmetry boundary. Such defects act from both the left and the right and hence $\mathbf{B}(D_4)$ is a bimodule 3-category over $\mathscr{C}$. All the possible bimodule 3-categories over $\mathscr{C}$ form a fusion 4-category $\text{\bf Bimod}(\mathscr{C})$, which can be identified with the fusion 4-category formed by topological defects of the 5d SymTFT.
    \label{fig:Almi}}
\end{figure}

Let $\mathscr{C}$ be a fusion 3-category. In this section, we will keep in mind a specific physical setup where $\mathscr{C}$ makes an appearance. The physical setup that of the SymTFT $\mathfrak{Z}(\mathscr{C})$ associated to $\mathscr{C}$ viewed as a symmetry of 4d QFTs. The SymTFT is a 5d TFT whose topological defects form the fusion 4-category $\Sigma\cZ(\mathscr{C})$. Such a TFT admits topological boundary conditions such that the topological defects living on such a boundary form the fusion 3-category $\mathscr{C}$. We pick such a boundary condition $\mathfrak{B}^{\text{sym}}_\mathscr{C}$ of $\mathfrak{Z}(\mathscr{C})$. See figure \ref{fig:Almi}.

We now consider the fusion 4-category $\mathbf{Bimod}(\mathscr{C})$ of $\mathscr{C}$-bimodule 3-categories. Physically, a $\mathscr{C}$-bimodule 3-category describes the 3-dimensional topological ends along $\mathfrak{B}^{\text{sym}}_\mathscr{C}$ of a 4-dimensional topological defect in the SymTFT $\mathfrak{Z}(\mathscr{C})$. See figure \ref{fig:Almi}. 
In fact, the bimodule 3-category completely characterizes this 4-dimensional topological defect, and the fusion 4-category $\mathbf{Bimod}(\mathscr{C})$ can be identified with the fusion 4-category formed by topological defects of the 5d TFT $\mathfrak{Z}(\mathscr{C})$:
\be\label{f1}
\mathbf{Bimod}(\mathscr{C})\simeq\mathbf{Mod}\big(\cZ(\mathscr{C})\big)
\ee
The \textbf{Brauer-Picard space} $\mathscr{B}r\mathscr{P}ic(\mathscr{C})$ of $\mathscr{C}$ consists of \textit{invertible} $\mathscr{C}$-bimodule 3-categories together with invertible bimodule higher morphisms. We define 
\be
BrPic(\mathscr{C}):=\pi_0(\mathscr{B}r\mathscr{P}ic(\mathscr{C}))
\ee
the \textbf{Brauer-Picard group} of $\mathscr{C}$. Physically, the group $BrPic(\mathscr{C})$ is the 0-form symmetry group of the 5d TFT $\mathfrak{Z}(\mathscr{C})$. Applying \eqref{f1} to the invertible objects, we obtain a group isomorphism
\be
BrPic(\mathscr{C})\cong Pic\big(\cZ(\mathscr{C})\big)
\ee

The homotopy groups of $\mathscr{B}r\mathscr{P}ic(\mathscr{C})$ are given as follows: 
\be 
\begin{tabular}{|c|c|c|c|c|c|}
\hline
$\pi_0$ & $\pi_1$ & $\pi_2$ & $\pi_3$ & $\pi_4$\\
\hline \\[-1em]
$BrPic(\mathscr{C})$ & $Inv(\mathcal{Z}(\mathscr{C}))$ & $Inv(\Omega\mathcal{Z}(\mathscr{C}))$ & $Inv(\Omega^2\mathcal{Z}(\mathscr{C}))$ & $\mathbb{C}^{\times}$\\
\hline
\end{tabular}
\ee

\noindent The space $\mathscr{B}r\mathscr{P}ic(\mathscr{C})$ carries a group-like monoid structure given by the relative tensor product of corresponding bimodule 3-categories over $\mathscr{C}$. In particular, we can consider its delooping $\mathrm{B}\mathscr{B}r\mathscr{P}ic(\mathscr{C})$. Physically, the homotopy group $\pi_i$ captures invertible topological defects of codimension-$(i+1)$ in the 5d TFT $\mathfrak{Z}(\mathscr{C})$.

For our purposes, we will focus on the case $\mathscr{C}:=\mathbf{3Vect}(A[1])$, with $A$ a finite abelian group. Let us spell out the higher homotopy groups of the associated Brauer-Picard space.

\begin{Proposition}\label{prop:BrPic3VectA[1]}
The Brauer-Picard space of $\mathbf{3Vect}(A[1])$, has homotopy groups given by: 
\be 
\begin{tabular}{|c|c|c|c|c|c|}
\hline
$\pi_0$ & $\pi_1$ & $\pi_2$ & $\pi_3$ & $\pi_4$\\
\hline \\[-1em]
$BrPic(\cVect(A[1]))$ & $0$ & $A\oplus \widehat{A}$ & $0$ & $\mathbb{C}^{\times}$\\
\hline
\end{tabular}
\ee
\noindent Moreover, the second $k$-invariant of the Brauer-Picard space of $\mathbf{3Vect}(A[1])$ is the class of $\varsigma$ in $H^6((A\oplus \widehat{A})[3];\mathbb{C}^{\times})$ corresponding to the canonical  sylleptic form on $A\oplus \widehat{A}$.
\end{Proposition}

\begin{proof}
For the homotopy groups $\pi_{\geq 1}$, recall that we have $\mathbf{3Vect}(A[1])\simeq \Sigma\mathbf{2Vect}(A[0])$. The first equivalence of sylleptic fusion 2-categories below then follows from \cite[section IV.B]{JF}: 
\be 
\Omega\mathcal{Z}(\Sigma\mathbf{2Vect}(A[0]))\simeq \mathcal{Z}_{(2)}(\mathbf{2Vect}(A[0]))\simeq \mathbf{2Vect}^\varsigma((A\oplus \widehat{A})[0]).
\ee 
The second equivalence follows by direct calculation. This yields the desired description of $\pi_{\geq 2}$, as well as that of the second Postnikov class. It therefore only remains to argue that $\pi_1$ vanishes. Categorically speaking this unpacks to the statement that the only Morita invertible $A$-graded fusion 1-category is $\mathbf{Vect}$. But, it is well-known that there is a unique Morita invertible fusion 1-category (as we are working over an algebraically closed field of characteristic zero), and this concludes the proof.
\end{proof}

Let us now fix a finite group $G$. The space $\mathscr{B}r\mathscr{P}ic(\mathscr{C})$ is related to faithfully $G$-graded extensions of $\mathscr{C}$ via the following general result, which also works for fusion $n$-category with $n$ arbitrary.\footnote{A much more general version of this result was announced in \cite{JFR3}. We refer the reader to \cite{ENO2} for the classical version of this result in the theory of fusion 1-categories. A variant for fusion 2-categories was obtained in \cite{D11}, and the proof therein can be adapted to accommodate for fusion 3-categories.}

\begin{Theorem}
\label{thm:alpha}
There is a bijection between faithfully $G$-graded extensions of $\mathscr{C}$ and homotopy classes of maps 
\be
\mathrm{B}G\rightarrow \mathrm{B}\mathscr{B}r\mathscr{P}ic(\mathscr{C}).
\ee
\end{Theorem}
Physically, such a map picks a $G$ 0-form symmetry of the 5d SymTFT $\mathfrak{Z}(\mathscr{C})$. Note that not only do we pick topological codimension-1 defects generating the $G$ 0-form symmetry, but we also pick possible intersections between such defects, thus picking a way to couple the background fields for the 0-form symmetry to the underlying theory $\mathfrak{Z}(\mathscr{C})$. By the above theorem, picking such a 0-form symmetry is equivalent to picking a $G$-graded fusion 3-category whose trivial grade is the fusion 3-category $\mathscr{C}$. The topological defects in grade $g\in G$ arise at the end of the codimension-1 topological defect $\mathbf{D}_4^g$ of $\mathfrak{Z}(\mathscr{C})$ generating the 0-form symmetry, along the symmetry boundary $\mathfrak{B}^{\text{sym}}_\mathscr{C}$.

Using Proposition \ref{prop:BrPic3VectA[1]}, it follows that in the specific case $\mathscr{C}=\mathbf{3Vect}(A[1])$, homotopy classes of maps can be described quite explicitly using classical obstruction theory from algebraic topology. Namely, they consist of the following data and conditions:
\begin{enumerate}
    \item A homomorphism of groups $G\rightarrow BrPic(\cVect(A[1]))$ such that a certain obstruction class in $\Theta\in H^4(G[1];A\oplus \widehat{A})$ vanishes. The vanishing of obstruction class $\Theta$ ensures that it does not form a non-trivial 3-group with the $A\oplus\widehat{A}$ 2-form symmetry of $\mathfrak{Z}(\mathscr{C})$. Recall that the 3-group can be expressed in terms of the background fields by a condition of the form
    \be
    \delta B_3=A_1^*\Theta
    \ee
    where $B_3$ is background field for 2-form symmetry and $A_1^*\Theta$ is the pullback to spacetime of the ``Postnikov class'' $\Theta$.
    \item A class $\alpha$ in (a torsor over) $H^3(G[1];A\oplus\widehat{A})$ such that a certain obstruction class in $H^6(G[1];\mathbb{C}^{\times})$ vanishes. 
    Physically, this is the symmetry fractionalization class, and the above requirement ensures that the $G$ 0-form symmetry does not carry any 't Hooft anomaly.
    \item A class $\tau$ in (a torsor over) $H^5(G[1];\mathbb{C}^{\times})$. Physically this can be understood as a choice of SPT phase for $G$ 0-form symmetry.
\end{enumerate}
\noindent The groups $H^3(G[1];A\oplus\widehat{A})$ and $H^4(G[1];A\oplus\widehat{A})$ above are twisted cohomology groups because the data $G\rightarrow BrPic(\cVect(A[1]))$ provides us with an action of $G$ on $A\oplus\widehat{A}\cong Inv(\Omega\mathcal{Z}(\mathbf{3Vect}(A[1])))$ via postcomposition with the canonical map $BrPic(\cVect(A[1]))\rightarrow Aut^{br}(\mathcal{Z}(\mathbf{3Vect}(A[1])))$.\footnote{This map is given by the conjugation action. For fusion 1-categories, this is explained in detail in \cite[section 8.2]{ENO2}.}

\begin{Remark}\label{rem:BrPicandWitt}
We give a more explicit description of $BrPic(\cVect(A[1]))$ in  section \ref{section:crossedbraided}. 
% to reflect the fact that this group parametrizes the additional (2+1)d TFTs that one can stack onto the (2+1)d interface in a (3+1)d theory.
\end{Remark}

In general, understanding the above data can be quite challenging. The main problem being that it is very difficult to compute explicitly the obstructions. As in lower categorical dimensions \cite{ENO2,D11}, we will focus on cases when the groups $H^4(G[1];A\oplus \widehat{A})$ and $H^6(G[1];\mathbb{C}^{\times})$ vanish. For instance, the group $H^6(G[1];\mathbb{C}^{\times})$ is trivial provided that $G$ is a finite cyclic group. Further, the group $H^4(G[1];A\oplus \widehat{A})$ vanishes if $G=\mathbb{Z}/2$, and the induced action on $A\oplus \widehat{A}$ swaps the two factors.

%%%%%%%%%%%%%%%%%%%%%%%%%%%%%%%%%%%%%%%%%%%
\subsection{3-Categories and Duality Symmetry}\label{subsection:3Catduality}
%%%%%%%%%%%%%%%%%%%%%%%%%%%%%%%%%%%%%%%%%%%

In this subsection we will study the homomorphisms described in Theorem \ref{thm:alpha}, for the groups 
\be
G\in\{\mathbb{Z}/2,\mathbb{Z}/4\}\,,
\ee
which are relevant for 3-categorical duality symmetries of (3+1)d quantum field theories.

\subsubsection{Extensions with $G= \Z/2$: $\mathbf{3TY}$-Categories}
We first begin with the case $G=\mathbb{Z}/2$ and consider $\Z/2$-graded extensions of fusion 3-categories. Tambara-Yamagami 3-categories are examples of such fusion 3-categories:
\begin{Definition}
A Tambara-Yamagami 3-category is a faithfully $\mathbb{Z}/2$-graded fusion 3-category $\mathscr{C}:=\mathscr{C}_0\boxplus\mathscr{C}_1$ such that \begin{enumerate}
    \item We have $\mathscr{C}_0\simeq \mathbf{3Vect}(A[1])$ for some finite abelian group $A$.
    \item The finite semisimple 3-category $\mathscr{C}_1$ is connected i.e.\ $\pi_0(\mathscr{C}_1)$ is a singleton, and invertible as a $\mathscr{C}_0$-bimodule.
\end{enumerate}
\end{Definition}

\begin{Definition}
An extended Tambara-Yamagami 3-category is a faithfully $\mathbb{Z}/4$-graded fusion 3-category $\mathscr{C}:=\mathscr{C}_0\boxplus\mathscr{C}_1 \boxplus\mathscr{C}_2 \boxplus\mathscr{C}_3 $ such that \begin{enumerate}
    \item We have $\mathscr{C}_0\simeq \mathbf{3Vect}(A[1])$ for some finite abelian group $A$.
    \item The finite semisimple 3-categories $\mathscr{C}_1$ and $\mathscr{C}_3$ are connected.
\end{enumerate}
\end{Definition}

We now motivate the above definition of a Tambara Yamagami 3-category. The first axiom above categorifies (twice) the condition that the $0$-graded part of a Tambara-Yamagami 1-category is of the form $\mathbf{Vect}(A[0])$. More generally, one could require that $\mathscr{C}_0\simeq \mathbf{3Vect}(\mathfrak{G})$ for some finite 3-group $\mathfrak{G}$. However, as for the case of fusion 2-categories studied in detail in \cite{DY23}, this imposes strong conditions on the structure of the 3-group $\mathfrak{G}$, and the most interesting case is $\mathfrak{G}=A[1]$ as we will consider. As for the second axiom, it categorifies the statement that the $1$-graded part of a Tambara-Yamagami 1-category is of the form $\mathbf{Vect}$. But $\mathbf{Vect}$ is the simple object of $\mathbf{2Vect}$, so that, when generalizing to fusion 3-categories, it therefore makes sense to consider any simple object of $\mathbf{4Vect}$, i.e.\ any connected finite semisimple 3-category. In fact, it follows from \cite{D9} that every connected finite semisimple 3-category is of the form $\Sigma^2\mathcal{B}$ for some braided fusion 1-category $\mathcal{B}$, that is, is twice connected. We will use this fact repeatedly below.

% When considering extensions of the fusion 3-category $\cVect(A[1])$ with $A=\Z/2$, the most physically relevant grading group is $\Z/2$ as this is a global symmetry of the associated SymTFT \cite{Kaidi:2022cpf}. We will focus on $\Z/2$-graded extensions of $\cVect(\Z/2[1])$, of which 

The possibilities for $\mathscr{C}_1$ are parametrized by finite semisimple $\cVect(A[1])$-bimodule 3-categories which are invertible and in particular have order 2. We now present two lemmas regarding such invertible 3-categories:

\begin{Lemma}
An invertible $\mathbf{3Vect}(A[1])$-bimodule structure on $\Sigma^2\mathcal{B}$ exists if $\mathcal{B}$ is non-degenerate.
\end{Lemma}
\begin{proof}
Let $\mathcal{B}$ be a non-degenerate braided fusion 1-category. Assume that the finite semisimple 3-category $\Sigma^2\mathcal{B}$ is equipped with a left $\mathbf{3Vect}(A[1])$-module structure. Upgrading this module structure to a bimodule one is equivalent to supplying a monoidal functor 
\be
\mathbf{3Vect}(A[1])\rightarrow \mathbf{End}_{\mathbf{3Vect}(A[1])}(\Sigma^2\mathcal{B})\,.
\ee
Moreover, the corresponding bimodule is invertible if and only if this monoidal functor is an equivalence. It is therefore enough to show that 
\be
\mathbf{End}_{\mathbf{3Vect}(A[1])}(\Sigma^2\mathcal{B})\simeq \mathbf{3Vect}(A[1]) \,.
\ee
In fact, the left-hand side is connected, so that it is enough to identify the associated braided fusion 2-category with $\mathbf{2Vect}(A[0])$. Now, it follows from a 3-categorical analogue of \cite[Lemma 3.2.1]{D9} that 
\be
\Omega\mathbf{End}_{\mathbf{3Vect}(A[1])}(\Sigma^2\mathcal{B})\simeq \mathcal{Z}(\Sigma\mathcal{B},\mathbf{2Vect}(A[0]))\,.
\ee
the centralizer of the image of $\mathbf{2Vect}(A[0])$ in $\mathcal{Z}(\Sigma\mathcal{B})\simeq \Omega\mathbf{End}(\Sigma^2\mathcal{B})$. But, we have $\mathcal{Z}_{(2)}(\mathcal{B})\simeq \mathbf{Vect}$, then $\mathcal{Z}(\Sigma\mathcal{B})\simeq \mathbf{2Vect}$ by \cite[Lemma 2.16]{JFR}. This implies that $\mathcal{Z}(\Sigma\mathcal{B},\mathbf{2Vect}(A[0]))\simeq \mathbf{2Vect}(A[0])$, which proves the claim, and completes the proof of the lemma. 
\end{proof}

The decategorified version of the next result is given in \cite[section 9.2]{ENO2}, by which our proofs are inspired. In terms of notations, let $q$ be a quadratic form on $A\times A$, that is $q:A\times A\rightarrow \mathbb{C}^{\times}$ is a function such that $q(-a,-b) = q(a,b)$ and the associated symmetric function 
\be
\mathsf{Bil}(q)((a,b),(c,d)):=\frac{q(a+c,b+d)}{q(a,b)q(c,d)}
\ee
is bilinear. Of particular interest to us is the (not necessarily symmetric) bilinear form $\mathsf{Bil}(q)_{12}(a,b):= \mathsf{Bil}(q)((a,0),(0,b))$.

\begin{Lemma}\label{lem:quadform}
Let $\mathcal{B}$ be a non-degenerate braided fusion 1-category. Invertible $\mathbf{3Vect}(A[1])$-bimodule structures on $\Sigma^2\mathcal{B}$ are classified by a subset of classes in $H^4((A\oplus A)[2];\mathbb{C}^{\times})$. We can identify such classes with quadratic forms $q:A\times A\rightarrow \mathbb{C}^{\times}$. The condition for invertibility is that the bilinear form $\mathsf{Bil}(q)_{12}$ associated to the quadratic form $q$ is non-degenerate in the sense that the induced map $A\rightarrow \widehat{A}$ given by $a\mapsto \mathsf{Bil}(q)_{12}(a,-)$ is an isomorphism.
\end{Lemma}
\begin{proof}
Firstly, given that $\mathcal{B}$ is non-degenerate, we have $\mathbf{End}(\Sigma^2\mathcal{B})\simeq\mathbf{3Vect}$, so that the data of a $\mathbf{3Vect}(A[1])$-bimodule structures on $\Sigma^2\mathcal{B}$ corresponds precisely to the data of a monoidal functor 
\be
\mathbf{3Vect}(A[1])\boxtimes \mathbf{3Vect}(A[1])\simeq\mathbf{3Vect}((A\oplus A)[1])\rightarrow \mathbf{3Vect},
\ee
which in turn is exactly the data of a class in $H^4((A\oplus A)[2];\mathbb{C}^{\times})$ as desired.

We have to compute $\mathbf{End}_{\mathbf{3Vect}(A[1])}(\Sigma^2\mathcal{B})$ together with its associated $\mathbf{3Vect}(A[1])$-bimodule structure. Recall from the previous lemma that we have 
\be 
\mathbf{End}_{\mathbf{3Vect}(A[1])}(\Sigma^2\mathcal{B})\simeq \mathbf{3Rep}(A[1])\simeq \mathbf{3Vect}(\widehat{A}[1])\,.
\ee
But, following \cite[remark 9.2]{ENO2}, one can compute that the induced $\mathbf{3Vect}(A[1])$-bimodule structure on $\mathbf{3Vect}(\widehat{A}[1])$ is given at the level of 1-morphisms by 
\be
\ba
(a\otimes \lambda)(x)&=\mathsf{Bil}(q)_{12}(a,x)\lambda(x)\cr  (\lambda\otimes b)(x) &=\lambda(x)\mathsf{Bil}(q)_{12}(b,x)\,,
\ea
\ee
for $\lambda\in \widehat{A}$ representing a 1-morphism in $\mathbf{3Vect}(\widehat{A}[1])$, and $a,b\in A$ representing 1-morphisms in $\mathbf{3Vect}(A[1])$, and $x\in A$. This shows that 
\be
\mathbf{End}_{\mathbf{3Vect}(A[1])}(\Sigma^2\mathcal{B})\simeq \mathbf{3Vect}(A[1])
\ee
as $\mathbf{3Vect}(A[1])$-bimodules where both sides are endowed with the canonical bimodule structures if and only if the assignment $A\rightarrow \widehat{A}$ given by $a\mapsto \mathsf{Bil}(q)_{12}(a,-)$ is an isomorphism.
\end{proof}

Now we study the conditions under which such an invertible bimodule 3-category has order two:

\begin{Lemma}\label{lem:order2}
An invertible $\mathbf{3Vect}(A[1])$-bimodule 3-category $\Sigma^2\mathcal{B}$ with $\mathcal{B}$ non-degenerate and corresponding quadratic form $q$ with $\mathsf{Bil}(q)_{12}$ non-degenerate has order two in\\ $BrPic(\mathbf{3Vect}(A[1]))$ if and only if  $q$ is symmetric and $\mathcal{B}$ has order two in the Witt group.
\end{Lemma}
\begin{proof}
An invertible $\mathbf{3Vect}(A[1])$-bimodule 3-category has order two in the Brauer-Picard group if and only if it is equivalent to its dual. But, the underlying 3-category of the dual is $\Sigma^2\mathcal{B}^{rev}$, and an equivalence of 3-categories $\Sigma^2\mathcal{B}\simeq \Sigma^2\mathcal{B}^{rev}$ is exactly a Witt equivalence between $\mathcal{B}$ and $\mathcal{B}^{rev}$, i.e.\ $\mathcal{B}$ has order (at most) two in the Witt group. Then, the $\mathbf{3Vect}(A[1])$-bimodule structure on $\Sigma^2\mathcal{B}^{rev}$ is given by the opposite of the monoidal functor $\mathbf{3Vect}(A[1])\boxtimes \mathbf{3Vect}(A[1])^{mop}\rightarrow \cVect$ providing $\Sigma^2\mathcal{B}$ with its bimodule structure. By inspection, the corresponding quadratic form is $(a,b)\mapsto q(b,a)$, and the result follows.
\end{proof}

\begin{Example} {\bf 3TY Categories for $G=\mathbb{Z}/2$.}
\label{ex:3TYZ2}\\
We give an example of a $\Z/2$-graded extension of $\mathbf{3Vect}(\Z/2[1])$ that is a TY 3-category. We first note that the two groups in which the obstructions live, namely
$H^6(\Z/2[0]; \mathbb{C}^\times)$ and $H^4(\Z/2[0], \underline{\Z/2 \oplus \Z/2})$ both vanish. The underline represents the fact that the coefficients are twisted by the $\Z/2$ action which swaps the two factors. This comes from the fact that the $\Z/2$ global symmetry of the SymTFT acts by permuting the two fields in the DW-action, and the map $\Z/2 \to Aut^{br}(\cZ(\cVect(A[0])))$ picks this action. There is a possible choice of a class $\tau$ in $H^5(\Z/2[0];\mathbb{C}^\times)=\Z/2$, which is analogous to the Frobenious Schur indicator familiar in the context of Tambara-Yamagami 1-categories. Additionally, there is a potential choice of class in $H^3(\Z/2[0]; \underline{\Z/2 \oplus \Z/2})=0$, but this group in fact vanishes.

Invertible $\mathbf{3Vect}(\Z/2[1])$-bimodule 3-categories are classified by (the Witt class of) a non-degenerate braided fusion 1-category $\mathcal{B}$ that has order at most two in the Witt group together with a class in 
\be
H^4((\Z/2 \oplus \Z/2)[2];\mathbb{C}^\times) = \Z/4 \oplus \Z/4 \oplus \Z/2 \,.
\ee
These are a total of $4\times4\times2=32$ number of quadratic forms. The conditions in Lemma \ref{lem:quadform} imply that only 16 quadratic forms are viable. Furthermore reducing to the case when the invertible bimodule has order two, we see that only 4 of the 16 remain  by Lemma \ref{lem:order2}, and can therefore be used to construct TY 3-categories. 

We can twist bimodules by tensor automorphisms of $\mathbf{3Vec}(\Z/2[1])$, analogous to \cite[Example 9.4]{ENO1}. Applying the twists allows us to relate the four bimodules. The group that labels the twist is given by $Aut(\mathbf{3Vec}(\Z/2[1])) \simeq H^4(\Z/2[2];\mathbb{C}^\times) = \Z/4$, and this group allows us to relate each of the four $\mathbf{3Vec}(\Z/2[1])$-bimodules with the others. 

% Furthermore, the general case is then recovered by stacking on $\mathbf{Mod}(\mathbf{Mod}(\mathcal{B}))$ to the nontrivial $\Z/2$-graded part of the 3-category as in Remark \ref{rem:BrPicandWitt}.

We can thus label such TY 3-categories as
\be
3\TY^w_{\tau,\sigma}(\Z/2[1]) \,,
\ee
where $w$ is a Witt class of order at most two, $\tau\in H^5(\Z/2[0];\mathbb{C}^\times)=\Z/2$ and a class $\sigma$ in a certain torsor for $H^4(\Z/2[2];\mathbb{C}^\times) = \Z/4$.

Note that for the physics construction in \cite{Kaidi:2022cpf} the choice of trivial Witt class $w=0$ for $\mathcal{B}$ was made. This corresponds to choosing $\mathcal{B}=\cZ(\mathcal{C})$ for a fusion category $\mathcal{C}$. Thus, the 3TY category describing such duality defects is of the form $3\TY^{0}_{\tau,\sigma}(\Z/2[1])$.

The duality defects (i.e. objects in $\mathscr{C}_1$) in $3\TY^{w\neq0}_{\tau,\sigma}(\Z/2[1])$ are obtained by stacking 3d TFTs in order two Witt class $w$ on top of duality defects in $3\TY^{w=0}_{\tau,\sigma}(\Z/2[1])$. More specifically, if $\mathbf{D}\in\mathscr{C}_1\subset3\TY^{0}_{\tau,\sigma}(\Z/2[1])$ is a topological defect in the 4d QFT, and $\mathfrak{T}(\mathcal{B})$ is a 3d TFT whose topological lines form a non-degenerate braided fusion category $\mathcal{B}$ lying in Witt class $w$, then the topological defect 
\be
\mathbf{D}\boxtimes \mathfrak{T}(\mathcal{B})\in\mathscr{C}_1\subset3\TY^{w}_{\tau,\sigma}(\Z/2[1])
\ee
that is, the topological defect $\mathbf{D}\boxtimes \mathfrak{T}$ lies in the non-trivial graded component of another 3TY category $3\TY^{w}_{\tau,\sigma}(\Z/2[1])$. The fact that $\mathcal{B}$ is of exact order two means that
\be
\mathcal{B}\neq\mathcal{Z}(\mathcal{C})
\ee
for any fusion category $\mathcal{C}$, but
\be
\mathcal{B}\boxtimes\mathcal{B}=\mathcal{Z}(\mathcal{C})
\ee
for some fusion category $\mathcal{C}$.
An example of such a $\mathcal{B}$ lying in a Witt class of order two is the square of the semion category.

\end{Example}

%%%%%%%%%%%%%%%%%%%%%%%%%%%%%%%%%%%%%%%%%%%
\subsubsection{Extensions with $G=\Z/4$: Generalized $\mathbf{3TY}$-Categories}
%%%%%%%%%%%%%%%%%%%%%%%%%%%%%%%%%%%%%%%%%%%
When $A=\Z/n$ for $n>2$, we need to consider $\Z/4$-graded extensions of $\cVect(A[1])$, unlike the Tambara-Yamagami categories that arise when $n=2$. In order to see this, recall that any $G$-graded extension of $\cVect(A[1])$ induces an
action of the group $G$ on the center $\cZ(\cVect(A[1])) = \Mod(\bVect^\varsigma(A\oplus \widehat{A}[1]))$ by braided monoidal equivalences.  In particular, we get an induced action of $G$ on $A\oplus\widehat{A}$ compatible with the canonical sympletic form $\varsigma$ on $A\oplus\widehat{A}$. Now, if the abelian group $A$ has elements of order strictly greater than $2$, the action of the group $G=\mathbb{Z}/2$ given by the swap $\begin{pmatrix}
    0 & +1\\
    +1 & 0
\end{pmatrix}$\,, which is associated to classical Tambara-Yamagami extensions, does not preserve this symplectic structure, and is therefore not an admissible action. Instead, we may consider the action of $G=\mathbb{Z}/4$ given by the matrix $\begin{pmatrix}
    0 & +1\\
    -1 & 0
\end{pmatrix}$\,. As explained in section \ref{subsection:3TYphysics}, from the perspective of DW-theory, we see that the group of actions is $\Z/4$ and this will be our focus for physical purposes. We will expand on \cite[Section 7]{Kaidi:2022cpf} by describing the categories that appears in each $\Z/4$-graded component. We will also describe surface operators in the component graded by $2\in\Z/4$ that have not been discussed in the literature.

Before getting into this study, we prove a useful lemma when used in conjunction with the 3-categorical version of Ostrik's theorem (see \cite{O1,D3} for lower categorical versions). This will be important for helping us with the constructions of the $\Z/4$-graded categories.

\begin{Lemma}\label{lem:BrtoPic}
Let $A$ be a finite abelian group, and let $\mathcal{B}_1$, $\mathcal{B}_2$ be two $A$-graded braided fusion 1-categories. Then, $\mathbf{Bimod}_{\mathbf{3Vect}(A[1])}(\Sigma\mathcal{B}_1,\Sigma\mathcal{B}_2)$ is the 3-category of $A$-graded $\mathcal{B}_1$-$\mathcal{B}_2$-central multifusion 1-categories, with the morphisms as in \cite[sections 3.3-6]{BJS}. Moreover, it is connected, and can be expressed as 
\be 
\mathbf{Bimod}_{\mathbf{3Vect}(A[1])}(\Sigma\mathcal{B}_1,\Sigma\mathcal{B}_2)\simeq \Sigma\mathbf{Mod}_{\mathbf{2Vect}(A[0])}(\mathcal{B}_1\boxtimes\mathcal{B}_2^{rev})\,.
\ee 
If the canonical $A$-grading on $\mathcal{B}_1\boxtimes\mathcal{B}_2$ is faithful, then \be 
\mathbf{Bimod}_{\mathbf{3Vect}(A[1])}(\Sigma\mathcal{B}_1,\Sigma\mathcal{B}_2)\simeq \Sigma^2(\mathcal{B}_1\boxtimes\mathcal{B}^{rev}_2)_0,
\ee 
where $(\mathcal{B}_1\boxtimes\mathcal{B}^{rev}_2)_0$ is the $0$-graded component of $\mathcal{B}_1\boxtimes\mathcal{B}_2$.
\end{Lemma}
\begin{proof}
The first part is an exercise in unfolding the definitions. In particular, connectedness is a direct consequence of the fact that every $A$-graded $\mathcal{B}_1$-$\mathcal{B}_2$-central fusion 1-category $\mathcal{C}$ can be viewed as a bimodule between itself and $\mathcal{B}_1\boxtimes\mathcal{B}_2$. Then, the middle part follows from the equivalences of fusion 2-categories \begin{align*}\mathbf{End}_{\Sigma\mathcal{B}_1-\Sigma\mathcal{B}_2}(\Sigma\mathcal{B}_1\boxtimes\Sigma\mathcal{B}_2)&\simeq \mathbf{Hom}_{\mathbf{3Vect}(A[1])}(I,\Sigma\mathcal{B}_1\boxtimes\Sigma\mathcal{B}_2)\\&\simeq \mathbf{Mod}_{\mathbf{2Vect}(A[0])}(\mathcal{B}_1\boxtimes\mathcal{B}_2),\end{align*} where the first equivalence comes from the free-forgetful adjunction, and the second follows from the first part. The last statement follows from the observation that, under the assumption that the canonical $A$-grading on $\mathcal{B}_1\boxtimes\mathcal{B}_2$ is faithful, then $\mathbf{Mod}_{\mathbf{2Vect}(A[0])}(\mathcal{B}_1\boxtimes\mathcal{B}_2)$ is a connected fusion 2-category.
\end{proof}

Now we return to the problem of understanding $\Z/4$ extensions. As we have seen above, invertible bimodule structures on $\mathbf{3Vect}$ are classified by quadratic forms $q:A\times A\rightarrow \mathbb{C}^{\times}$ such that the associated bilinear form $\mathsf{Bil}(q)_{12}$ is non-degenerate. As a consequence of the above discussion, we are interested in understanding for which quadratic forms $q$, the invertible $\mathbf{3Vect}(A[1])$-bimodule $\mathbf{3Vect}^q$ has order $4$ in the group $BrPic(\mathbf{3Vect}(A[1]))$.

\begin{Lemma}\label{lem:propertiesq}
Let $A=\Z/n$ with $n>2$ be a cyclic group. Let $\mathbf{3Vect}^q$ be an invertible $\mathbf{3Vect}(A[1])$-bimodule 3-category where $q$ is a quadratic form with $q(a,0) = q_1(a) $ and $q(0,b) = q_2(b)$ both trivial, i.e.\ the subgroups $A\oplus 0$ and $0\oplus A$ are both Lagrangian. The $\mathbf{3Vect}(A[1])$-bimodule $\mathbf{3Vect}^q$ has order $4$ in $BrPic(\mathbf{3Vect}(A[1]))$ if and only if $\mathsf{Bil}(q)_{12}$ is anti-symmetric non-degenerate.
\end{Lemma}

\begin{proof}
We start by computing \be \mathbf{3Vect}^q\boxtimes_{\mathbf{3Vect}(A[1])}\mathbf{3Vect}^q
\ee
as a plain 3-category. This computation only depends on $q(a,0)=q_1(a)$ and $q(0,b)=q_2(b)$, which correspond to the left and right $\cVect(A[1])$-module structures on $\cVect$ respectively. Let $\mathbf{Vect}_{A \xrightarrow{=}A}^{q_1}$ and $\mathbf{Vect}_{A \xrightarrow{=}A}^{q_2}$ be braided separable algebras in $\bVect(A)$ by means of the canonical $A$-gradings given by $A \xrightarrow{=}A$.
Then, we have 
\be\ba
\mathbf{3Vect}^q &\simeq\mathbf{Mod}(\mathbf{LMod}_{\mathbf{2Vect}(A)}(\mathbf{Vect}_{A \xrightarrow{=}A}^{q_1}))
\cr
\mathbf{3Vect}^q&\simeq\mathbf{Mod}(\mathbf{Mod}_{\mathbf{2Vect}(A)}(\mathbf{Vect}_{A \xrightarrow{=}A}^{q_2}))
\ea
\ee
as left and right $\mathbf{3Vect}(A[1])$-module 3-categories respectively. We then have  
\begin{align}
\mathbf{3Vect}^q\boxtimes_{\mathbf{3Vect}(A[1])}\mathbf{3Vect}^q &\simeq \mathbf{Mod}(\mathbf{Bimod}_{\mathbf{2Vect}(A[0])}(\mathbf{Vect}_{A\xrightarrow{=}A},\mathbf{Vect}_{A\xrightarrow{=}A}^{q_2}))\\ &\simeq \mathbf{3Rep}(A[1])\simeq \mathbf{3Vect}(A[1]),\end{align} 
as both $q_1$ and $q_2$ were assumed to be trivial. Therefore $\mathbf{3Vect}^q\boxtimes_{\mathbf{3Vect}(A[1])}\mathbf{3Vect}^q$ is a quasi-trivial bimodule, i.e.\ it arises from a monoidal autoequivalence of $\mathbf{3Vect}(A[1])$.
Then, the action on the center induced by $\mathbf{3Vect}^q$ is given by $\begin{pmatrix}
    0 & \widehat{\mathsf{Bil}(q)_{12}}^{-1}\\
    \mathsf{Bil}(q)_{12} & 0
\end{pmatrix}$\,, where $\mathsf{Bil}(q)_{12}:A\rightarrow \widehat{A}$ is given by $a\mapsto \mathsf{Bil}(q)_{12}(a,-)$. This automorphism of $A\oplus\widehat{A}$, which preserves the canonical symplectic form, has order at exactly $4$ if and only if $\mathsf{Bil}(q)_{12} = -\widehat{\mathsf{Bil}(q)_{12}}$, i.e.\ the bilinear form $\mathsf{Bil}(q)_{12}$ is antisymmetric. It therefore only remains to argue that the corresponding invertible bimodule 3-category $\mathbf{3Vect}^q$ has order exactly $4$ in $BrPic(\mathbf{3Vect}(A[1]))$. To see this, it is enough to show that $\mathbf{3Vect}^q\boxtimes_{\mathbf{3Vect}(A[1])}\mathbf{3Vect}^q$ has order two. But, this last invertible bimodule is quasi-trivial, so that its order is exactly the order of the induced action on the center. 
To see the action on the center we note that the center has a pairing given by $\varsigma$, and a quasi-bimodule gives a monoidal action on $A[1]$. This induces a nontrivial braided action on the $A[0]$ part of the center via the description in \eqref{eq:centerequiv}, and the pairing $\varsigma$ gives an action on the other part of the center, which is detectable on the $A[0]$ part. 
Furthermore, the action of $\mathbf{3Vect}^q\boxtimes_{\mathbf{3Vect}(A[1])}\mathbf{3Vect}^q$ on the center is given by $\begin{pmatrix}
    -1 & 0\\
    0 & -1
\end{pmatrix}$\,, which has order two as desired.
\end{proof}

 With the knowledge of the bimodule structures that the nontrivial graded components of the $\Z/4$-graded category should have, we now write the generalization of Tambara-Yamagami 3-categories: 
\begin{equation}\label{eq:Z4graded}
    \mathbf{3Vect}(A[1]) \boxplus \mathbf{3Vect}^q \boxplus {\mathbf{3Vect}^{\phi}(A[1])}\boxplus \mathbf{3Vect}^{q^{-1}}
\end{equation}
where $\cVect(A[1])$ lies in degree 0, and $\cVect^{q^{-1}}$ lies in degree 3.  We introduce the data
$\phi =(\widetilde{\phi},\alpha) \in  H^4(A[2];\mathbb{C}^\times) \rtimes Aut(A)$ where $\phi$ represents a tensor automorphism of $\cVect(A[1])$,  given by the latter semi-direct product. Therefore, $\phi$ will act as a twist of the $\mathbf{3Vect}(A[1])$ action so that it is not merely the degree zero component again: $\widetilde \phi$ modifies the module structure for $\cVect(A[1])$ on the degree 2 category, and $a \in A$, $a \mapsto -a$ under $\alpha$ because we are considering 3TY categories.  There is a two sided action on $\cVect(A[1])$ given by $\widetilde{\phi}$ and we will denote $\widetilde{\phi}_1$ for a left action and $\widetilde{\phi}_2$ for a right action.

We now describe the fusion rules for the $\Z/4$-graded category. We will represent a 3-dimensional topological defect corresponding to the object $\mathbf{Vect}$ inside of $\cVect^q$ as $\mathbf{D}^{(1)}_3$, and a 3-dimensional topological defect corresponding to the object $\mathbf{Vect}$ inside of $\cVect^{q^{-1}}$ as $\mathbf{D}^{(3)}_3$
 % denotes the fact that this operator is an object of the 3-categories in the $i$-th degree of the category, and $q$ is the additional information of the quadratic form as in equation \eqref{eq:Z4graded}. 
 % We include this label to serve as a reminder of the fact that the degree 2 component of the category has a $\mathbf{3Vect}(A[1])$-bimodule structure depending on $q$. 
 A 3-dimensional topological defect in degree $0$ is necessarily a condensation defect. The condensation defect corresponding to the object $\mathbf{Vect}(A[0])$ with canonical $A$-grading inside of $\cVect(A[1])$ is labeled as $\mathbf{C}^{(0)}_3$. A 3-dimensional topological defect corresponding to the object $\mathbf{Vect}(A[0])$ with canonical $A$-grading inside of $\cVect^\phi(A[1])$ is labeled as $\mathbf{D}^{(2)}_3$.
 % 3-dimensional topological defects in degree $2$ are fusions of condensation defects with the charge conjugation 0-form symmetry defect. The topological defect obtained by fusing $\mathbf{C}^{(0)}_3$ with charge conjugation is labeled as $\mathbf{C}^{(2)}_3$.
 The fusion rules of the $\mathbf{C}_3^{(0)}$ and $\mathbf{D}^{(2)}_3$ defects are:
 % More precisely, $\mathbf{C}^{(0)}_3$ is the object $\mathbf{Vect}(A[0])$ with canonical $A$-grading inside of $\cVect(A[1])$, and $\mathbf{D}^{(1),q}_3$ is the object $\mathbf{Vect}$ inside of $\cVect^q$.  While $q$ gives rise to $q_1$ and $q_2$, Lemma \ref{lem:propertiesq} necessitated $q_1$ and $q_2$ to be both trivial. 
 %In order for the fusion rules to be consistent among the nontrivial graded components we need $\Omega \widetilde{\phi}_i = \Omega q_i$ so that the actions on $\cVect(A[1])$ are compatible. 
% For ease of reading we will omit the extra symbols $\phi$ and $q$ when discussing the fusion rules between objects, and the left and right actions are left implicit. The fusion rules involving exclusively the objects $\mathbf{C}^{(i)}_3$ are:
 \be\label{eq:fusionof0and2}
 \ba
 \mathbf{C}^{(0)}_3 \times \mathbf{C}^{(0)}_3 
 &= \mathbf{D}^{(2)}_3 \times \mathbf{D}^{(2)}_3=\mathbf{Vect}({A}[0])\boxtimes \mathbf{C}^{(0)}_3\cr 
 \mathbf{C}^{(0)}_3 \times \mathbf{D}^{(2)}_3 &= \mathbf{Vect}^{\Omega \widetilde{\phi}_1}(A[0])\boxtimes \mathbf{D}^{(2)}_3\cr 
 \mathbf{D}^{(2)}_3 \times \mathbf{C}^{(0)}_3  &= \mathbf{Vect}^{\Omega \widetilde{\phi}_2}(A[0])\boxtimes \mathbf{D}^{(2)}_3 \,.
\ea
\ee 
Here $\mathbf{Vect}(A[0])$ and $\mathbf{Vect}^{\Omega \widetilde{\phi}_i}(A[0])$ can be thought of as  ``TQFT coefficients" \cite{Roumpedakis:2022aik,Copetti:2023mcq}, as they are taken with the trivial $A$-grading, and $\Omega \widetilde{\phi}$ is the element in $H^3(A[1];\mathbb{C}^\times)$ which suspends to $\widetilde{\phi}$. This gives the associator for the TQFT coefficients. Since $\mathbf{C}^{(0)}$ and $\mathbf{C}^{(2)}$ correspond to $\mathbf{Vect}({A}[0])$ with canonical $A$-grading in the higher Morita category, their product is given by $\mathbf{Vect}((A \oplus A)[0])$ with the product grading, which can be factorized into a trivially $A$-graded component, and a canonically $A$-graded component by taking a basis that consists of $(a,-a)$ and $(a,0)$ in $A \oplus A$. These are analogous to the fusion rules presented in \cite[Section 3]{Antinucci:2022cdi} and \cite[Equation 1.4]{Choi:2022zal} which should be matched with our case $\widetilde\phi=0$. In this case, we can identify $\mathbf{D}_3^{(2)}$ as the fusion $\mathbf{C}_3^{(0)}\times\mathcal{U}$, where $\mathcal{U}$ is the charge conjugation operator. 
% In particular, we can match with the fusion in \cite{Choi:2022zal} in the following way: $\mathbf{C}_3^{(0)}$ is the condensation defect $\mathcal{C}_0$, while $\mathbf{C}_3^{(2)}$ is the tensor product $\mathcal{C}_0\times\mathcal{U}$, where $\mathcal{U}$ is the charge conjugation operator. 

The fusion rules only consisting of $\mathbf{D}^{(i)}_3$ are given by 
\be
\ba
\mathbf{D}^{(1)}_3\times \mathbf{D}^{(1)}_3&= \mathbf{D}^{(3)}_3 \times \mathbf{D}_3^{(3)}=\mathbf{D}^{(2)}_3\cr 
\mathbf{D}^{(1)}_3\times \mathbf{D}^{(3)}_3 &= \mathbf{D}^{(3)}_3\times \mathbf{D}^{(1)}_3= \mathbf{C}^{(0)}_3\,.
\ea
\ee
We also have the fusion rules 
\be\ba
\mathbf{C}^{(0)}_3\times \mathbf{D}^{(1)}_3
& = \mathbf{Vect}^{}(A[0])\boxtimes \mathbf{D}^{(1)}_3\cr 
\mathbf{D}^{(1)}_3 \times \mathbf{C}^{(0)}_3& = \mathbf{Vect}^{}(A[0])\boxtimes \mathbf{D}^{(1)}_3
\cr 
\mathbf{D}^{(2)}_3\times \mathbf{D}^{(1)}_3& = \mathbf{Vect}^{}(A[0])\boxtimes \mathbf{D}^{(1)}_3\cr  \mathbf{D}^{(1)}_3\times \mathbf{D}^{(2)}_3 &= \mathbf{Vect}^{}(A[0])\boxtimes \mathbf{D}^{(1)}_3\,,
\ea
\ee
which are the same if we replace $\mathbf{D}^{(1)}_3$ by $\mathbf{D}^{(3)}_3$. We note that unlike in \eqref{eq:fusionof0and2}, there is no associator for the TQFT coefficient since both $q_1$ and $q_2$ are trivial.

%%%%%%%%%%%%%%%%%%%%%%%%%%%%%%%%%%%%%%%%%%%
\section{Crossed Braided Higher Fusion Categories}\label{section:crossedbraided}
%%%%%%%%%%%%%%%%%%%%%%%%%%%%%%%%%%%%%%%%%%%
This section aims to explain extensions from the point of view of the Picard space $\Pic(\mathscr{S})$ associated to a braided fusion 3-category $\mathscr{S}$. We will be particularly interested in the case 
\be 
\mathscr{S}=\mathcal{Z}(\cVect(A[1]))\,.
\ee
This perspective will be complementary to that, based the Brauer-Picard space $\BrPic(\cVect(A[1]))$, which was used in the previous section to study extensions of $\cVect(A[1])$. From a physical point of view, this is obtaining the Drinfeld center/SymTFT for the $G$-extensions of fusion 3-categories (in particular those involving duality defects) by gauging a $G$ 0-form symmetry of the SymTFT associated to the fusion 3-category lying in the trivial grade. In our physics applications, the trivial grade is occupied by $\cVect(A[1])$, and the SymTFT before $G$-gauging is a DW theory.

After some mathematical preliminaries, we will be able to explicitly compute the category that lives on the twist defects for the $G$ zero-form symmetry in the (4+1)d SymTFT. It follows from the mathematics that such a category has to satisfy certain conditions in order to be compatible with gauging the zero form symmetry, and we will describe the potential choices.

%%%%%%%%%%%%%%%%%%%%%%%%%%%%%%%%%%%%%%%%%%%
\subsection{Definition and Classification}\label{subsection:crossbraideddefs}
%%%%%%%%%%%%%%%%%%%%%%%%%%%%%%%%%%%%%%%%%%%

In order to motivate our definition of a $G$-crossed braided fusion 3-category, we first make some remarks about $G$-crossed braided 1-categories \cite[Section 8.24]{EGNO}. Recall that this flavour of highly structured fusion 1-categories appears in physics for instance in the context of global symmetries of (1+1)d chiral conformal field theories, but also of (2+1)d symmetry enriched phases \cite{Barkeshli:2014cna,Barkeshli:2016mew,Barkeshli:2019vtb}. It was shown in \cite{JPR} that the 4-category of $G$-crossed braided 1-categories is equivalent to the 4-category of 3-categories $\mathscr{C}$ equipped with a 3-functor $\mathrm{B}G \rightarrow \mathscr{C}$ that is full on objects and 1-morphisms.
Further, faithfully $G$-graded extensions correspond to functors $\mathrm{B}G \rightarrow \mathscr{C}$ such that $\mathscr{C}$ is \textit{connected} in the sense that there exists a non-zero 2-morphism between any two 1-morphisms.
It is possible to use this higher categorical perspective to recover the extension theory results from \cite{ENO2} (see \cite[Example 1.12]{JPR}). More precisely, let us consider a 3-functor $\mathrm{B}G \rightarrow \mathscr{C}$ that is full on objects and 1-morphisms. Equivalently, we may view this data as a monoidal 2-functor $G\rightarrow \mathfrak{C}$, where $\mathfrak{C}=\Omega\mathscr{C}$. This describes a $G$-crossed braided fusion 1-category if and only if $\mathfrak{C}$ is a locally finite semisimple 2-category. But, if $\mathfrak{C}$ is both locally finite semisimple and connected, then $Cau(\mathfrak{C})$ is equivalent to $\mathbf{Mod}(\mathcal{B})$, where $\mathcal{B}=\Omega\mathfrak{C}$ a braided fusion 1-category. Now, any monoidal 2-functor $\pi: G\rightarrow \mathbf{Mod}(\mathcal{B})$ must factor through the Picard space $\Pic(\mathcal{B})$. Putting this discussion together, we do recover the fact that faithfully $G$-graded extensions of $\mathcal{B}$ are parameterized by (homotopy classes of) maps $\mathrm{B}G\rightarrow \mathrm{B}\Pic(\mathcal{B})$.
In particular, in the special case when $\pi:G\rightarrow \mathbf{Mod}(\mathcal{B})$, the corresponding $G$-crossed braided extension $\mathcal{A}$ of $\mathcal{B}$ is given by  
\be
\mathcal{A}=\bigoplus_{g\in G} \pi(g), \quad \pi(e)=\cB\,,
\ee
exactly as in \cite{ENO2}.

The above discussion motivates the following definition.

\begin{Definition}
    A $G$-crossed braided 3-category is a 5-category $\mathbf{C}$ equipped with a 5-functor from $\mathrm{B}G \rightarrow \mathbf{C}$ that is full on objects and 1-morphisms \footnote{This definition readily generalizes to $n$-categories and the subsequent proof carries over to finite semisimple $n$-categories.}.
\end{Definition}
Physically such a $G$-crossed braided 3-category describes topological 3-dimensional defects in a $G$ 0-form symmetric 5d topological boundary of a $G$ 0-form symmetric 6d TFT whose $G$-symmetric topological boundary conditions form the 5-category $\mathbf{C}$. The functor $\mathrm{B}G \rightarrow \mathbf{C}$ at the level of objects picks such a topological boundary $\mathfrak{B}$. At the level of 1-morphisms it picks a collection of topological 4-dimensional defects $U_g$ in $\mathfrak{B}$ labeled by $g\in G$. The $g$-grade of the $G$-crossed braided 3-category describes topological 3-dimensional ends of $U_g$.

 In fact, we will need to make a further restriction:\ As is familiar in the theory of graded categories, we want to focus our attention on \textit{faithfully} graded categories. For $G$-crossed braided fusion 1-categories, under the equivalence of \cite{JPR}, this is implemented by requiring that the corresponding monoidal 2-category $\mathfrak{C}$ be connected, i.e.\ there exists a non-zero 1-morphism between any two objects.

\begin{Definition}
    A $G$-crossed braided 3-category is faithfully graded if the 5-category $\mathbf{C}$ is connected, that is, there exists a non-zero 2-morphism between any two 1-morphisms.
\end{Definition}

 \begin{Proposition}
    Let $\mathscr{S}$ be a braided fusion 3-category. Faithfully graded $G$-crossed braided extensions of $\mathscr{S}$ are classified by homotopy classes of maps $\mathrm{B}G\rightarrow \mathrm{B}\mathscr{P}ic(\mathscr{S})$.
 \end{Proposition}
\begin{proof}
We fix a faithfully graded $G$-crossed braided fusion 3-category, that is, a 5-category $\mathbf{C}$ equipped with a 5-functor from $F:\mathrm{B}G \rightarrow \mathbf{C}$ that is full on objects and 1-morphisms, such that the $Hom$-3-categories of $\mathbf{C}$ are all finite semisimple and non-zero. Equivalently, we may recast the 5-functor $F$ as a monoidal 4-functor $G \rightarrow \Omega\mathbf{C}$ where $\mathscr{S}$ is the braided fusion 3-category corresponding to the trivially graded component of the crossed braided extension, i.e.\ $\mathscr{S}=End_{\Omega\mathbf{C}}(F(e))$. Then, we can consider the Cauchy completion $Cau(\Omega\mathbf{C})$ of $\Omega\mathbf{C}$. By definition, we have that $Cau(\Omega\mathbf{C})$ is a finite semisimple 4-category. But, because $\mathbf{C}$ is connected, we have $Cau(\Omega\mathbf{C})\simeq \Sigma\mathscr{S}=\mathbf{Mod}(\mathscr{S})$. It then follows that $\mathbf{Mod}(\mathscr{S})^{\times} = \mathscr{P}ic(\mathscr{S})$, so that the functor $F$ provides us with a map of spaces $\mathrm{B}G\rightarrow \mathrm{B}\mathscr{P}ic(\mathscr{S})$, that is precisely the desired extension data. The above procedure may also be run in reverse, concluding the proof.
\end{proof}

The space $\mathscr{P}ic(\mathscr{S})$ is the Picard space of $\mathscr{S}$, and we also write $Pic(\mathscr{S}):=\pi_0(\mathscr{P}ic(\mathscr{S}))$, the Picard group of $\mathscr{S}$. The homotopy groups of $\mathscr{P}ic(\mathscr{S})$ are given as follows: \be
\begin{tabular}{|c|c|c|c|c|c|}
\hline
$\pi_0$ & $\pi_1$ & $\pi_2$ & $\pi_3$ & $\pi_4$\\
\hline \\[-1em]
$Pic(\mathscr{S})$ & $Inv(\mathscr{S})$ & $Inv(\Omega\mathscr{S})$ & $Inv(\Omega^2\mathscr{S})$ & $\mathbb{C}^{\times}$\\
\hline
\end{tabular}
\ee
In particular, if we take $\mathscr{S}:=\mathcal{Z}(\cVect(A[1]))=\Sigma \mathbf{2Vect}^\varsigma(A \oplus \widehat{A}[0])$, it follows from proposition \ref{prop:BrPic3VectA[1]} that we have: 
\be
\begin{tabular}{|c|c|c|c|c|c|}
\hline
$\pi_0$ & $\pi_1$ & $\pi_2$ & $\pi_3$ & $\pi_4$\\
\hline \\[-1em]
$Pic(\mathscr{S})$ & $0$ & $A\oplus \widehat{A}$ & $0$ & $\mathbb{C}^{\times}$\\
\hline
\end{tabular}
\ee
Physically speaking, the $n$-th homotopy groups give the $(4-n)$-dimensional topological operators. By definitions, there are no lines and by non-degeneracy all 3d membranes are condensation defects. Therefore, both $\pi_1$ and $\pi_3$ are trivial, while $\pi_2$ gives the surface defects in the Dijkgraaf-Witten theory.

For later use, we unpack the data of a homotopy class of map $\mathrm{B}G\rightarrow \mathrm{B}\mathscr{P}ic(\Sigma \mathbf{2Vect}^\varsigma(A \oplus \widehat{A}[0]))$. They consist of the following data and conditions:
\begin{enumerate}
    \item A map of groups $G\rightarrow Pic(\Sigma \mathbf{2Vect}^\varsigma(A \oplus \widehat{A}[0]))$ such that a certain obstruction class in $H^4(G[1];A\oplus \widehat{A})$ vanishes.
    \item An class $\alpha$ in (a torsor over) $H^3(G[1];A\oplus \widehat{A})$ such that a certain obstruction class in $H^6(G[1];\mathbb{C}^{\times})$ vanishes.\footnote{The group homomorphism $G\rightarrow Pic(\Mod( \mathbf{2Vect}^\varsigma(A \oplus \widehat{A}[0])))$ induces an action of $G$ on $A\oplus\widehat{A}$, and the first cohomology group is interpreted with this action.}
    \item A class $\tau$ in (a torsor over) $H^5(G[1];\mathbb{C}^{\times})$.
\end{enumerate}
The physical interpretation of these were given already in section \ref{sec:BPS}. 

Finally, we comment on the precise relation between the two extension theories that we have been discussing. Let $\mathscr{S}$ be a braided fusion 3-category, and $G$ a finite group. We generalize \cite{ENO2} to $G$-crossed braided extensions of $\mathscr{S}$ and show that they are parameterised by homotopy classes of maps \be
\mathrm{B}G\rightarrow \mathrm{B}\mathscr{P}ic(\mathscr{S})\,.
\ee
In the special case when $\mathscr{S}=\mathcal{Z}(\mathscr{C})$ is the Drinfeld center of the fusion 3-category $\mathscr{C}$, then it follows from a 3-categorical version of \cite{GNN} that the $G$-equivariantization of any $G$-crossed braided extension of $\mathcal{Z}(\mathscr{C})$ is a Drinfeld center. More precisely, in this case, the space $\mathrm{B}\mathscr{P}ic(\mathcal{Z}(\mathscr{C}))$ is equivalent to $\mathrm{B}\mathscr{B}r\mathscr{P}ic(\mathscr{C})$, and, as we have seen above, the latter parameterises group-graded extensions of the fusion 3-category $\mathscr{C}$. Therefore, given a map $\mathrm{B}G\rightarrow \mathrm{B}\mathscr{P}ic(\mathcal{Z}(\mathscr{C}))$, we can consider the corresponding $G$-crossed braided extension $\mathscr{A}$ of $\mathcal{Z}(\mathscr{C})$, as well as the $G$-graded extension $\mathscr{D}$ of $\mathscr{C}$. Then, we have that the equivariantization or \textit{gauging} of $\mathscr{A}$ by its $G$-action is equivalent as a braided fusion 3-category to the Drinfeld center $\mathcal{Z}(\mathscr{D})$.

%%%%%%%%%%%%%%%%%%%%%%%%%%%%%%%%%%%%%%%%%%%
\subsection{Twisted Graded Witt}\label{subsection:twistedwitt}
%%%%%%%%%%%%%%%%%%%%%%%%%%%%%%%%%%%%%%%%%%%

We wish to understand the structure of the space $\mathrm{B}\Pic(\mathcal{Z}(\cVect(A[1])))$, and especially its fundamental group. We will do so at a slightly broader level of generality by considering $A$ an arbitrary finite abelian group and $s$ a syllepsis on $A$, so that $\mathbf{2Vect}^s(A[0])$ is a sylleptic strongly fusion 2-categories.\footnote{More generally, we could also consider the case of $\mathbf{2Vect}^{(b,s)}(A[0])$ when the braiding $b$ is also allowed to be non-trivial.} It follows from the results of section \ref{sec:HigherMorita} that the group $Pic(\Sigma\mathbf{2Vect}^s(A[0]))$ is the Witt group of (suitable invertible) braided $A$-graded fusion 1-categories, whose Deligne tensor product is twisted by the syllepsis $s$. More precisely, we have \be
Pic(\Sigma\mathbf{2Vect}^s(A[0])) \cong \pi_0(\mathsf{Mor}_2^{sep}(\mathbf{2Vect}^s(A[0])))
\ee
as groups, and below we will unpack the definition of the right hand-side in the spirit of \cite{DMNO,BJSS}.

We begin by spelling out some of the structure of the bare 4-category $\mathsf{Mor}_2^{sep}(\mathbf{2Vect}^s(A[0]))$. We observe that this does not depend on the syllepsis $s$. Firstly, objects are by definition $A$-graded braided multifusion 1-categories. Secondly, by a slight generalization of \cite{BJS}, 1-morphisms are given by $A$-graded central multifusion 1-categories. In order to make this precise, we use the following definitions.

\begin{Definition}
Let $\mathcal{B}$ be an $A$-graded braided multifusion 1-category. An $A$-graded $\mathcal{B}$-central multifusion 1-category is an $A$-graded multifusion 1-category $\mathcal{C}$ equipped with a braided functor $F:\mathcal{B}\rightarrow\mathcal{Z}(\mathcal{C})$ that is compatible with the $A$-gradings.
\end{Definition}

\begin{Definition}
Let $\mathcal{C}$ be a $\mathcal{B}$-central $A$-graded multifusion 1-category. We write $\mathcal{Z}_{(2)}(\mathcal{C},\mathcal{B})$ for the centralizer of the image of $\mathcal{B}$ in $\mathcal{Z}(\mathcal{C})$, which is an $A$-graded braided multifusion 1-category.
\end{Definition}

\noindent A 1-morphism $\mathcal{B}_1\nrightarrow\mathcal{B}_2$ in $\mathsf{Mor}_2^{sep}(\mathbf{2Vect}^s(A[0]))$ between two $A$-graded braided multifusion 1-categories $\mathcal{B}_1$ and $\mathcal{B}_2$ is an $A$-graded $\mathcal{B}_1$-central multifusion 1-category $\mathcal{C}$ equipped with a braided functor $\mathcal{B}_2^{rev}\rightarrow \mathcal{Z}_{(2)}(\mathcal{C},\mathcal{B}_1)$ that is compatible with the $A$-gradings.\footnote{It may be tempting to succinctly rewrite this definition as the data of a $\mathcal{B}_1\boxtimes\mathcal{B}_2^{rev}$-central multifusion 1-category. However, because we will consider twisted Deligne tensor products below, we prefer not to use plain Deligne tensor products so as to avoid confusions.} For our purposes, it is crucial to understand when two objects of $\mathsf{Mor}_2^{sep}(\mathbf{2Vect}^s(A[0]))$ are equivalent. This is spelled out in the following definition.

\begin{Definition}
Two $A$-graded braided fusion 1-categories $\mathcal{B}_1$ and $\mathcal{B}_2$ are $A$-Witt equivalent if there exists a $\mathcal{B}_1$-central $A$-graded fusion 1-category $\mathcal{C}$ with fully faithful braided functor $F:\mathcal{B}\rightarrow \mathcal{Z}(\mathcal{C})$ together with a grading preserving equivalence 
\be
\mathcal{Z}_{(2)}(\mathcal{C},\mathcal{B}_1)^{rev}\simeq \mathcal{B}_2\,.
\ee
We say that a braided $A$-graded fusion 1-category is $A$-Witt trivial if it is $A$-Witt equivalent to $\mathbf{Vect}$.
\end{Definition}

\noindent The next lemma follows from the definition, or may also be proved directly as in \cite{DMNO}.

\begin{Lemma}
The notion of being $A$-Witt equivalent is an equivalence relation on $A$-graded braided multifusion 1-categories.
\end{Lemma}

We will now begin to describe the monoidal structure on the 4-category $\mathsf{Mor}_2^{sep}(\mathbf{2Vect}^s(A[0]))$. This is an $s$-twisted variant of the familiar Deligne tensor product. More precisely, given two braided $A$-graded multifusion 1-categories $\mathcal{B}_1$ and $\mathcal{B}_2$ with braidings $\beta_1$ and $\beta_2$ respectively, we set \be
\mathcal{B}_1\boxtimes_A^{s}\mathcal{B}_2:= \mathcal{B}_1\boxtimes\mathcal{B}_2
\ee
with $A$-grading given by the sum of the gradings on $\mathcal{B}_1$ and $\mathcal{B}_2$, with obvious monoidal structure, and with braiding given on simple objects $B_1,C_1\in\mathcal{B}_1$ and $B_2,C_2\in\mathcal{B}_2$ by \be
\beta(B_1\boxtimes B_2, C_1\boxtimes C_2):= \beta_1(B_1,C_1)\boxtimes\beta_2(B_2,C_2)\cdot s(\mathrm{gr}(B_2),\mathrm{gr}(C_1))\,,
\ee
where $\mathrm{gr}(B_2), \mathrm{gr}(C_1)\in A$ denote the gradings of these simple objects.

We will also need the $s$-twisted variant of the reverse of a braided $A$-graded fusion 1-category. More precisely, given a braided $A$-graded fusion 1-category $\mathcal{B}$ with braiding $\beta$, the underlying $A$-graded multifusion 1-category of $\mathcal{B}^{s-rev}$ is simply given by $\mathcal{B}$, the braiding is given on simple objects $B,C\in\mathcal{B}$ by 
\be
\beta^{s-rev}(B,C):= \beta(C,B)^{-1}\cdot s(\mathrm{gr}(B),\mathrm{gr}(C))\,.
\ee

\begin{Definition}
We say that a braided $A$-graded fusion 1-category $\mathcal{B}$ is $s$-$A$-Witt invertible, or $s$-invertible for short, if $\mathcal{B}\boxtimes_A^s\mathcal{B}^{s-rev}$ is $A$-Witt trivial.
\end{Definition}

\begin{Remark}
The idea behind the above definition is that $\mathcal{B}$ is an object of a monoidal 4-category that has duals. Therefore, if it has an inverse, this inverse must coincide with its dual.
\end{Remark}

\begin{Definition}
The Witt group $\Witt(A,s)$ of $s$-invertible braided $A$-graded fusion 1-categories is the quotient of the monoid of $s$-invertible braided $A$-graded fusion 1-categories under the relation of $A$-graded Witt equivalence.
\end{Definition}

The next result follows from our definitions.

\begin{Proposition}
The Witt group $\Witt(A,s)$ is isomorphic to $Pic(\mathbf{2Vect}^s(A[0]))$.
\end{Proposition}

\subsubsection{Pointed Twisted Witt Group}

We will not attempt to describe the complete structure of the group $\Witt(A,s)$. Instead, we will focus on the subgroup $\Witt^{pt}(A, s)$ generated by the $s$-invertible pointed braided $A$-graded fusion 1-categories. The data of a pointed braided $A$-graded fusion 1-category is completed captured by the notion of an $A$-graded premetric group, and we will spell out below what the condition of being $s$-invertible amounts to. We begin by recalling two definitions from \cite[Appendix A]{DGNO} for the reader's convenience.

\begin{Definition}\label{deg:quadraticform}
Let $G$ be a finite abelian group. A quadratic form on $G$ with value in $\mathbb{C}^{\times}$ is a function $q:G\rightarrow \mathbb{C}^{\times}$ such that $q(g) = q(g^{-1})$ for every $g\in G$, and the associated symmetric function 
\be
\mathsf{Bil}(q)(G[1];h):=\frac{q(gh)}{q(g)q(h)}
\ee 
is bilinear.\footnote{Sometimes this condition is also referred to being a \textit{bicharacter}.}
\end{Definition}
 For any quadratic form $q:G\rightarrow \mathbb{C}^{\times}$, any integer $n$, and any element $g\in G$, we have 
 \be
 q(n\cdot g) = q(g)^{n^2}\,.
 \ee

 \begin{Definition}
A premetric group is a pair $(G[1];q)$ consisting of a finite abelian group $G$ and a quadratic form $q$. A metric group is a premetric group $(G[1];q)$ whose associated bilinear form $\mathsf{Bil}(q)$ is non-degenerate.
\end{Definition}

\begin{Definition}
An $A$-graded premetric group is a finite abelian group $G$ equipped with a quadratic form $q$, and a group homomorphism $f:G\rightarrow A$. We write $(G_0,q_0)$ for the premetric group $Ker(f)$ equipped with the obvious quadratic form $q_0=q|_{G_0}$.
\end{Definition}

Let $\mathcal{B}$ be a braided pointed $A$-graded fusion 1-category, i.e. an $A$-graded fusion 1-category equipped with a braiding, such that all simple objects are invertible.
The corresponding $A$-graded premetric group is $Inv(\mathcal{B})$, the group of invertible objects of $\mathcal{B}$, equipped with the quadratic form 
\be
q(B):= \beta(B,B)\cdot s(\mathrm{gr}(B),\mathrm{gr}(B)).
\ee

\begin{Remark}
A conceptual reason for the above definition is that in a sylleptic monoidal 2-category, the only reasonable way of trivializing the double braiding is the syllepsis. But such a trivialization is needed for the definition of a quadratic form.
\end{Remark}

The following lemma is clear

\begin{Lemma}
Pointed braided $A$-graded fusion 1-categories are classified by $A$-graded premetric groups.
\end{Lemma}

\begin{Definition}
 An $A$-graded metric group is an $A$-graded premetric group whose underlying premetric group is metric, i.e.\ the quadratic form is non-degenerate.
\end{Definition}

\begin{Remark}
Let $(f:G\rightarrow A,q)$ be an $A$-graded metric group. Then $(G_0,q_0)$ does not have to be a metric group. However, we do have the following lemma.
\end{Remark}

\begin{Lemma}\label{lem:0metricdecomposition}
Let $(f:G\rightarrow A,q)$ be an $A$-graded premetric group. If $H\subseteq G_0$ is such that $q|_H$ is non-degenerate, then there exists a decomposition of $A$-graded premetric groups 
\be 
(f:G\rightarrow A,q)\cong (H\xrightarrow{0} A,q_H)\oplus (f:H^{\bot}\rightarrow A, q|_{H^{\bot}})
\ee 
where $H^{\bot}$ is the orthogonal complement of $H$ in $G$ with respect to the bilinear form $\mathsf{Bil}(q)$.
\end{Lemma}

\begin{Definition}
Let $(f:G\rightarrow A,q)$ be an $A$-graded metric group. A subgroup $H\subseteq G_0$ is called isotropic if $q|_H=\mathsf{triv}$. It is called Lagrangian if, in addition, $H^{\bot} = H$.
\end{Definition}

\begin{Definition}
An $A$-graded premetric group $(f:G\rightarrow A,q)$ is $A$-trivial if it is a metric group, i.e.\ $q$ is non-degenerate, and there exists a Lagrangian subgroup $H\subseteq G_0$.
\end{Definition}

The following lemma follows easily from the fact that forgetting the grading defines a non-monoidal 4-functor $\mathsf{Mor}_2^{sep}(\mathbf{2Vect}^s(A[0]))\rightarrow \mathsf{Mor}_2^{sep}(\mathbf{2Vect})$ together with the corresponding well-known statement in the case $A=0$ \cite{DMNO}.

\begin{Lemma}\label{lem:pointedAWitttrivial}
A pointed $A$-graded braided fusion 1-category $\mathcal{B}$ is $A$-Witt trivial if and only if the corresponding $A$-graded premetric group is $A$-trivial.
\end{Lemma}

We now explain how to multiply $A$-graded premetric groups together, so as to be able to recover the subgroup structure on $\Witt^{pt}(A, s)\subseteq \Witt(A, s)$.

Let $(f:G\rightarrow A,q_G)$ and $(k:H\rightarrow A,q_H)$ be two $A$-graded premetric groups. Their $s$-twisted product 
\be 
(f:G\rightarrow A,q_G)\boxtimes_A^s(k:H\rightarrow A,q_H)
\ee
is the $A$-graded premetric group obtained by endowing the $A$-graded group 
\be 
G\oplus H\xrightarrow{f+k} A
\ee 
with the quadratic form 
\be 
Q(g;h):=q_G(g)\cdot q_H(h)\cdot s(f(g),k(h))\,.
\ee 
We also need to consider an $s$-twisted variant of the opposite of an $A$-graded premetric group $(f:G\rightarrow A,q)$. We set $$(f:G\rightarrow A,q)^{s\mathrm{-}op}:= (f:G\rightarrow A,\widetilde{q}),$$ where $\widetilde{q}(g):= q(g)^{-1}\cdot s(f(g),f(g))^2$.

\begin{Definition}
We say that an $A$-graded premetric group $(f:G\rightarrow A,q)$ is $s$-invertible if 
\be 
(f:G\rightarrow A,q)\boxtimes_A^s(f:G\rightarrow A,q)^{s\mathrm{-}op}
\ee 
is $A$-trivial.
\end{Definition}

\begin{Example}
Provided the syllepsis is the trivial one, a braided $A$-graded fusion 1-category $\mathcal{B}$ is invertible if and only if its underlying braided fusion 1-category is non-degenerate. Namely, in this case, forgetting the grading is compatible with the notion of being invertible, and it is well-known that a braided fusion 1-category is invertible if and only if it is non-degenerate \cite{DN, BJSS,JFR,D9}. Thus it only remains to show that a braided $A$-graded fusion 1-category $\mathcal{B}$ whose underlying braided fusion 1-category is non-degenerate is invertible. This follows from the observation that $\mathcal{B}\boxtimes_A\mathcal{B}^{rev}$ is $A$-Witt trivial as it is equivalent as a braided $A$-graded fusion 1-category to the Drinfeld center of $\mathcal{B}$. Further, if we take $A$ to be trivial, we recover exactly the Witt group $\Witt$ of non-degenerate braided fusion categories introduced in \cite{DMNO}.
\end{Example}

\begin{Definition}
The group $\Witt^{pt}(A,s)$ is the quotient of the monoid of $s$-invertible $A$-graded premetric group (with product given by $\boxtimes_A^s$) under the submonoid of $A$-Witt trivial $A$-graded premetric groups. 
\end{Definition}

In other words, two $s$-invertible $A$-graded premetric groups $(f:G\rightarrow A,q_G)$ and $(k:H\rightarrow A,q_H)$ define the same class in $\Witt^{pt}(A,s)$ if 
\be 
(f:G\rightarrow A,q_G)\boxtimes_A^s(k:H\rightarrow A,q_H)^{s\mathrm{-}op}
\ee 
is $A$-trivial. In particular, the next result follows from lemma \ref{lem:pointedAWitttrivial} together the definitions.

\begin{Lemma}
The group $\Witt^{pt}(A,s)$ is the subgroup of $\Witt(A,s)$ generated by the $s$-invertible pointed braided $A$-graded fusion 1-categories.
\end{Lemma}

%%%%%%%%%%%%%%%%%%%%%%%%%%%%%%%%%%%%%%%%%%%
\subsubsection{Computing with higher categorical methods}\label{subsubsection:highercatmethods}
%%%%%%%%%%%%%%%%%%%%%%%%%%%%%%%%%%%%%%%%%%%

Let $\mathfrak{S}$ be a sylleptic fusion 2-category. We record some general facts about the structure of the group $\Witt(\mathfrak{S})$, which are particularly useful in the case when $\mathfrak{S}$ is a sylleptic strongly fusion 2-category. Firstly, we have the following general result \cite[Proposition 2.3.2]{JF2} (yet a more general version was announced \cite{Reu}).

\begin{Theorem}\label{thm:fibersequence}
Let $\mathbf{C}$ be a fusion $n$-category. There is a fiber sequence 
\be 
\mathbf{C}^{\times}\rightarrow \mathscr{A}ut^{\otimes}(\mathbf{C})\rightarrow \mathbf{Bimod}(\mathbf{C})^{\times}\,.
\ee
\end{Theorem}

Since $\mathbf{Bimod}(\mathbf{C})\simeq \mathbf{Mod}(\mathcal{Z}(\mathbf{C}))$, the above fiber sequence simplifies if $\mathbf{C}$ is Morita invertible as a fusion $n$-category. Namely, this implies that $\mathbf{Mod}(\mathcal{Z}(\mathbf{C}))\simeq \mathbf{(n+1)Vect}$. We record the special case where $\mathbf{C} = \Sigma^2\mathbf{2Vect}^\varsigma((A\oplus \widehat{A})[0])$.
The next corollary follows from the above theorem, and offer valuable insight into the structure of the group $\Witt(A\oplus A,s)$.\footnote{We will drop the hat on the second factor of $A$ when doing computations, since we can choose a noncanonical isomorphism between $\widehat A$ and $A$.}

\begin{Corollary}\label{cor:shortexact}
There is an exact sequence of groups 
\be 
0\rightarrow \Witt\rightarrow \Witt(A\oplus A,s)\rightarrow Aut^{syp}(A\oplus A,\mathsf{Alt}(s))\ltimes H^5((A\oplus A)[3];\mathbb{C}^{\times})\rightarrow \mathbb{Z}/2\,,
\ee 
where $Aut^{syp}(A\oplus A,\mathsf{Alt}(s))$ is the group of automorphisms of $A\oplus A$ preserving the alternating 2-from $\mathsf{Alt}(s)$ introduced in section \ref{subsection:sylleptic}.
\end{Corollary}

\begin{proof}
The exact sequence of (non-abelian) groups above is obtained by unpacking theorem \ref{thm:fibersequence}. Firstly, the homotopy groups of $\mathbf{5Vect}^{\times}$ are known \cite{JF, JFR}. In particular, we have $\pi_1(\mathbf{5Vect}^{\times})\cong\mathcal{W}itt$ and $\pi_0(\mathbf{5Vect}^{\times})\cong\mathbb{Z}/2$. Secondly, it is easy to check that the map $\Witt\rightarrow \Witt(A\oplus A,s)$ is injective, and, in fact, also central. Finally, it follows from the fact that $\Sigma^2\mathbf{2Vect}^\varsigma((A\oplus \widehat{A})[0])$ is 2-connected that 
\be 
\mathscr{A}ut^{\otimes}(\Sigma^2\mathbf{2Vect}^\varsigma((A\oplus \widehat{A})[0]))\simeq \mathscr{A}ut^{syl}(\mathbf{2Vect}^\varsigma((A\oplus A)[0]))\,.
\ee 
Further, by inspecting the definition, we find that $\pi_0(\mathscr{A}ut^{syl}(\mathfrak{S}))\cong Aut^{syp}(A\oplus A,\mathsf{Alt}(s))\ltimes H^5((A\oplus A)[3];\mathbb{C}^{\times})$.
\end{proof}

\begin{Remark}
From the examples that we will examine in the next section, it appears that the map $Aut^{syp}(A\oplus A,\mathsf{Alt}(s))\ltimes H^5((A\oplus A)[3];\mathbb{C}^{\times})\rightarrow \mathbb{Z}/2$ is often $0$. In particular, it is not surjective in general.
\end{Remark}

We end this section by recording two conjectures. The first one is motivated by the fact that the Witt spaces commute with equivariantizations \cite{DHJFPPRNY}.

\begin{Conjecture}
For any finite abelian group $A$, there is an isomorphism of groups 
\be 
\Witt(A,\mathsf{triv})\cong \Witt \oplus H^4(A [2];\mathbb{C}^{\times}).
\ee
\end{Conjecture}

The subgroup $\Witt^{pt}(A,s)$ is much more amenable to explicit computations than $\Witt(A,s)$. Motivated by the last conjecture above, we make the following claim.

\begin{Conjecture}
For any abelian group $A$ and syllepsis $s$, the canonical map 
\be \Witt^{pt}(A,s)/\Witt^{pt}\rightarrow \Witt(A,s)/\Witt
\ee 
is an isomorphism.
\end{Conjecture}

%%%%%%%%%%%%%%%%%%%%%%%%%%%%%%%%%%%%%%%%%%%
\subsection{Computing Twisted Graded Witt}
%%%%%%%%%%%%%%%%%%%%%%%%%%%%%%%%%%%%%%%%%%%
\label{subsubsection:loworders}

%%%%%%%%%%%%%%%%%%%%%%%%%%%%%%%%%%%%%%%%%%%
%\subsubsection{Graded categories of low orders}
%%%%%%%%%%%%%%%%%%%%%%%%%%%%%%%%%%%%%%%%%%%

 We will take $A = \Z/2 \oplus \Z/2$ and write $a=(1,0)$ and $b=(0,1)$.\footnote{We have made a slight notational change in this subsection from earlier to simplify the notation: what used to be  $A \oplus \widehat{A}$, is replaced by  $A = \Z/2 \oplus \Z/2$. The current section is about $\mathbb{Z}_2$ 1-form symmetry. } We consider the non-trivial syllepsis $s$ given by 
 \be 
 s(a,b)=-1, s(a,a)=s(b,b)=s(b,a)=+1\,.
 \ee 
 We now study $\Witt^{pt}(\Z/2 \oplus \Z/2,s)/\Witt^{pt}$.
Thanks to our general considerations in section \ref{subsubsection:highercatmethods}, we have 
\be 
\Witt^{pt}(\mathbb{Z}/2\oplus \mathbb{Z}/2,s)/\Witt^{pt}\cong Ker \big(\mathsf{SL}(2,2)\ltimes (\mathbb{Z}/2\oplus \mathbb{Z}/2)\rightarrow \mathbb{Z}/2\big)\,,
\ee 
using the fact that $Aut^{syp}(\mathbb{Z}/2\oplus \mathbb{Z}/2,s)\cong \mathsf{SL}(2,2)\cong S_3$ and $H^5((\Z/2 \oplus \Z/2)[3];\mathds{k}^\times)=\Z/2 \oplus \Z/2$. It follows from the inclusion of symplectic groups $\mathbb{Z}/2\oplus 0\subset \mathbb{Z}/2\oplus \mathbb{Z}/2$ that the graded semion give an element of order 4 in $\Witt^{pt}(\mathbb{Z}/2\oplus \mathbb{Z}/2,s)/\Witt^{pt}$ and therefore the above map to $\Z/2$ must be zero and 
\be 
\Witt^{pt}(\mathbb{Z}/2\oplus \mathbb{Z}/2,s)/\Witt^{pt}\cong S_4\,.
\ee In particular, $\Witt^{pt}(\mathbb{Z}/2\oplus \mathbb{Z}/2,s)/\Witt^{pt}$ has elements of order $3$ and $4$.
 
In order to familiarize with some computations regarding the graded twisted Witt group, we will show which graded categories have rank 4, 3, and 2 in $\Witt^{pt}(\Z/2 \oplus \Z/2,s)/\Witt^{pt}$. 
For the purpose of understanding duality defects, we will only need to understanding the order two categories. A subset of the order two categories will be the square of order 4 categories, so we first present the computation to show what categories are order 4.  \\

\noindent\underline{\textbf{Elements of order 4:}}
we have the following graded groups: 
\be \ba 
\mathfrak{a}:=& (\Z/2 \xrightarrow{a} A;\, q(1)=\pm i)\cr  \mathfrak{b}:=&(\Z/2 \xrightarrow{b} A;\, q(1)=\pm i)\cr 
 \mathfrak{c}:=&(\Z/2 \xrightarrow{a+b} A;\, q(1)=\pm 1) \,.
 \ea
 \ee
If we square $\mathfrak{a}$, we get 
\be 
\mathfrak{a}^2= (\Z/2\oplus \Z/2 \xrightarrow{(a,a)} A;\, q(1,0)=\pm i,\, q(0,1)=\pm i,\, q(1,1)=-1)\,,
\ee 
where $q(1,1) = q(1,0)q(0,1)s(a,a)$, and the product is taken with respect to $\oplus^s_{A}$. Taking $\mathfrak{a}$ to the fourth power we get a grading 
$\Z/2\oplus \Z/2 \oplus \Z/2\oplus \Z/2 \xrightarrow{(a,a,a,a)} A$, and we now discuss the graded components. When this is understood we can show how this is trivial in the $s$-invertible $A$-graded Witt group, and thus $\mathfrak{a}$ has order 4. The 0-graded sector is generated by $\{(1,1,0,0), (0,1,1,0), (0,0,1,1)\}$, and we have a non-degenerate subcategory spanned by $(1,1,0,0)$ and $(0,1,1,0)$. We call the subcategory generated by these two elements $\mathfrak{z}=(\Z/2 \oplus \Z/2 \rightarrow A, q(1,0)=q(0,1)=q(1,1)=-1)$. Since we are working modulo $\Witt$, the non-degenerate part of the 0-graded sector of the category can be trivialized.
We then look at the orthogonal complement of $\mathfrak{z}$ and find that it is spanned by $(1,1,0,0)$ and $(0,0,1,1)$ 
%\matt{this doesnt seem right},
%which have trivial quadratic forms  and is Lagrangian since they braid nontrivially with the four elements in grade $a$.
Therefore, $(\mathfrak{a})^4$ is trivial, and the other graded categories can be checked in a similar way to be order 4. \\

\noindent\underline{\textbf{Elements of order 3:}}
In the group $S_4$, the product of any two non-commuting elements of order 4 gives an element of order 3. In fact, by considering their inverses and the products in the reverse order, we can obtain all elements of order 3 that way. The products $\mathfrak{a}^{\pm1}\mathfrak{b}^{\pm 1}$ and $\mathfrak{b}^{\pm1}\mathfrak{a}^{\pm 1}$ give: 
\be 
(\Z/2\oplus \Z/2 \xrightarrow{=} A;\, q(1,0)=\pm i, q(0,1)=\pm i, q(1,1)=\pm 1) \,.
\ee
This gives the 8 elements of order 3. One can also check that all of the other possible choices of two non-commuting elements of order 4 yield the same list of elements of order 3.

It is instructive to do an example working modulo $\Witt^{pt}$, for seeing how one of the above elements has order 3. Consider $(\mathfrak{a}^{+1} \mathfrak{b}^{+1})^3$, we first compute $\mathfrak{a} \mathfrak{b}$ which is 
\be
(\Z/2 \oplus \Z/2 \xrightarrow{(a,b)} A, q(1,0)=i, q(0,1)=i, q(1,1)=1)\,.
\ee 
Here, we used that $q(1,1) =q(1,0)q(0,1) s(a,b)$, and $s(a,b)=-1$. Taking the square we have the group $\Z/2\oplus \Z/2 \oplus \Z/2 \oplus \Z/2 \xrightarrow{(a,b,a,b)} A$, and we now discuss the graded components.  The 0-graded sector is generated by the maps $(a,b,a,b) = (1,0,1,0)$ and $(a,b,a,b) =(0,1,0,1)$. The quadratic forms are given by $q(1,0,1,0)= -1$, $q(0,1,0,1)=-1$ and $q(1,1,1,1) = (+1)(+1)s(ab,ab) = -1$. We see that the 0-graded sector is indeed non-degenerate, so we can split this off and move on to the $a$-graded and $b$-graded sectors. In the $a$-graded sector, we have 4 possible maps given by $\{(1,0,0,0), (0,0,1,0), (1,1,0,1), (0,1,1,1)\}$. One can check that the element $(0,1,1,1)$ braids trivially with $(1,0,1,0)$ and $(0,1,0,1)$, which are 0-graded. Similarly in the $b$-graded sector, we have 4 possible maps given by $\{ (0,1,0,0), (0,0,0,1), (1,1,1,0), (1,0,1,1)\}$, and the element $(1,1,1,0)$ braids trivially with $(1,0,1,0)$ and $(0,1,0,1)$. We therefore see that $(\mathfrak{a}^{+1} \mathfrak{b}^{+1})^2$ is given by 
\be 
(\Z/2 \oplus \Z/2 \xrightarrow{(a,b)}A, q'(0,1)=-i, q'(1,0)=-i, q'(1,1)=-1)
\,,
\ee
where $(0,1,1,1)$ and $(1,1,1,0)$ generates the two $\Z/2$s. We multiply again by $(\mathfrak{a}^{+1} \mathfrak{b}^{+1})$ and analyze the 0-graded sector. We denote the new quadratic form for the product by $\tilde{q}$ which takes values
\begin{align}
    \tilde{q}(1,0,1,0) = q(1,0) q'(1,0) = 1\,,\\ \notag 
    \tilde{q}(0,1,0,1) = q(0,1) q'(0,1) = 1\,, \\ \notag 
    \tilde{q}(1,1,1,1) = q(1,1) q'(1,1) s(a,b) = 1\,,
\end{align}
and therefore, is symmetric. We can check if this is Lagrangian by computing if $(1,0,0,0)$ and $(0,1,0,0)$ braid nontrivially with $(1,0,1,0)$ and $(0,1,0,1)$. A quick computate sees that indeed $(1,0,0,0)$ and $(0,1,0,0)$ braid nontrivially, and therefore $(\mathfrak{a}^{+1} \mathfrak{b}^{+1})^3$ is trivial. \\

\noindent\underline{\textbf{Elements of order 2:}}
There are three elements of order 2 which are squares of the categories that have order 4 (the  graded semion). The three categories are:
\be 
\ba 
(\Z/2\oplus \Z/2 &\xrightarrow{(a,a)} A;\, q(1,0)= i, q(0,1)= i, q(1,1)=-1) \cr 
(\Z/2\oplus \Z/2 &\xrightarrow{(b,b)} A;\, q(1,0)= i, q(0,1)= i, q(1,1)=-1)\cr 
(\Z/2\oplus \Z/2 &\xrightarrow{(a+b,a+b)} A;\, q(1,0)= 1, q(0,1)= 1, q(1,1)=-1)\,.
\ea
\ee
There is, in addition, 6 other elements of order 2. These can also be obtained by multiplying together elements of order 4 appropriately. A full list is given by 
\be 
\{\mathfrak{a}^{2}\mathfrak{b}^{\pm 1},\, \mathfrak{a}\mathfrak{b}^{\pm 1}\mathfrak{a},\, \mathfrak{b}\mathfrak{a}^{\pm 1}\mathfrak{b}\} \,.
\ee 
For instance, we can explicitly write:
\be 
\mathfrak{a}^{2}\mathfrak{b}^{+1} = (\Z/2\oplus \Z/2 \oplus \Z/2 \xrightarrow{(a,a,b)} A; Q) \,,
\ee
with quadratic form $Q$ given by 
\be \ba 
Q(1,0,0)&=Q(0,1,0)=Q(0,0,1)=+i, Q(1,1,0)=-1 \cr 
Q(1,0,1)&=Q(0,1,1)=1, Q(1,1,1)=-i \,.
\ea
\ee

We now give the order of these elements when we lift into $\Witt(\Z/2\oplus \Z/2, s)$.

\noindent\underline{\textbf{Lifting elements of order 4:}}
 The 6 elements we found before also have order 4 in $\Witt(\mathbb{Z}/2\oplus \mathbb{Z}/2,s)$. One can check that a metric group is picked up when computing the fourth power, but this metric group is Witt trivial so it causes no problem when we lift. \\

\noindent\underline{\textbf{Lifting elements of order 3:}}
When computing the order of $\mathfrak{a}^{\pm1}\mathfrak{b}^{\pm 1}$ and $\mathfrak{b}^{\pm1}\mathfrak{a}^{\pm 1}$
in $\Witt^{pt}(\Z/2 \oplus \Z/2, s)/ \Witt^{pt}$. We killed non-degenerate categories of the form 
\be 
\mathcal{C} :=(\Z/2\oplus \Z/2 \xrightarrow{0} A;\, q(1,0)=q(0,1)=q(1,1)=-1)\,,
\ee
which are order 2 in $\Witt^{pt}$. The `0' in the arrow denotes that fact that we are killing the grading in the category.
By the short exact sequence in Corollary \ref{cor:shortexact}, $(\mathfrak{a}^{\pm1}\mathfrak{b}^{\pm 1})\cC$ and $\mathfrak{b}^{\pm1}\mathfrak{a}^{\pm 1} \cC$ have order 6 in $\Witt^{pt}(\Z/2 \oplus \Z/2, s)$. Therefore, to lift the elements $\mathfrak{a}^{\pm1}\mathfrak{b}^{\pm 1}$ and $\mathfrak{b}^{\pm1}\mathfrak{a}^{\pm 1}$ to $\Witt^{pt}(\Z/2 \oplus \Z/2, s)$, we have to consider 
\be 
\{\mathfrak{a}^{\pm1}\mathfrak{b}^{\pm 1} \sqrt{\cC},\, \mathfrak{b}^{\pm1}\mathfrak{a}^{\pm 1} \sqrt{\cC}\}\,.
\ee
Since the category $\mathcal{C}$ corresponds to the category $\mathrm{SO}(8)_1$, which has order $\Z/8$ in $Ker(\mathcal{W}itt\rightarrow s\mathcal{W}itt)$, we can find its square root, which we denoted $\sqrt{\cC}$. \\

\noindent\underline{\textbf{Lifting elements of order 2:}}
By the analysis that can be done for lifting the elements of order 4, the three elements of order 2 which are squares of the graded semion lift without any contribution.

We now consider the lifting of the order 2 elements that are not the squares of the graded semions. The other elements in $\{\mathfrak{a}^{2}\mathfrak{b}^{\pm 1},\, \mathfrak{a}\mathfrak{b}^{\pm 1}\mathfrak{a},\, \mathfrak{b}\mathfrak{a}^{\pm 1}\mathfrak{b}\}$
acquire additional contribution from a non-degenerate category
\be 
\mathcal{C} :=(\Z/2\oplus \Z/2 \xrightarrow{0} A;\, q(1,0)=q(0,1)=q(1,1)=-1)
\ee
with order 2 in $\Witt^{pt}$
along with a Lagrangian 
\be 
\mathcal{L} :=(\Z/2\oplus \Z/2 \xrightarrow{0} A;\, q(1,0)=q(0,1)=q(1,1)=1).
\ee
As an example, we consider $(\mathfrak{a}^2\mathfrak{b}^{+1})^2$, for which we have the group 
\be \Z/2 \oplus \Z/2 \oplus \Z/2 \oplus \Z/2 \oplus \Z/2 \oplus \Z/2 \xrightarrow{(a,a,b,a,a,b)}A\,.
\ee
The 0-graded part generated by $(0,1,0,0,1,0)$ and $(0,0,1,0,0,1)$ is nondegenerate as one can check that $q(0,1,0,0,1,0)=-1$ and $q(0,0,1,0,0,1)=-1$ by using the definition of $\mathfrak{a}^2\mathfrak{b}^{+1}$. These two elements therefore generate each of the $\Z/2$s in $\mathcal{C}$. As we have already seen, we can find a square root of $\cC$, so the categories 
\be 
\mathcal{S}:=\{ \mathfrak{a}^{2}\mathfrak{b}^{\pm 1}\sqrt{\cC},\, \mathfrak{a}\mathfrak{b}^{\pm 1}\mathfrak{a}\sqrt{\cC},\, \mathfrak{b}\mathfrak{a}^{\pm 1}\mathfrak{b}\sqrt{\cC}\}\ee
have order two in $\Witt^{pt}(\Z/2 \oplus \Z/2, s)$. There are infinitely many choices of square roots that one can choose in $\Witt^{pt}$, and the proliferation of choices implies a freedom to dress the 3d defect with Morita invertible braided fusion categories.

\subsection{Constructing crossed braided 3-categories}
\label{subsub:examplesZ2}

%\subsubsection{Example: $\Z/2$-crossed braided extension}

Having very explicitly computed elements of order 2 in $\Witt^{pt}(\mathbb{Z}/2\oplus \mathbb{Z}/2,s)$, we attempt to understand the extension that will lead to $\mathbf{3}\TY[\Z/2]$ from the point of view of the (4+1)d, i.e. \!\!the bottom line in Figure \ref{fig:gauging}. 
We give the corresponding $\Z/2$-crossed braided extension. To completely gauge the zero-form global symmetry 
and obtain the SymTFT for $\mathbf{3}\TY[\Z/2]$
would also involve performing equivariantization for the corresponding crossed braided 3-categories, a step which we do not perform.

Let $\mathcal{B}$ be one of the $\Z/2 \times \Z/2$-graded braided fusion 1-category in $\Witt^{pt}(\mathbb{Z}/2\oplus \mathbb{Z}/2,s)$ of order 2 in the set $\mathcal{S}$, as described in the end of section \ref{subsubsection:loworders}. 
The other elements of order 2 arising from squares of the graded semion form a normal subgroup, which is the subgroup that lives in the image of $\Witt$.

The underlying 3-category corresponding to the $\Z/2$-graded extension of $ \Sigma \bVect^\varsigma((\Z/2\oplus \Z/2 )[0])$ has the form:
\begin{equation*}
    \mathscr{B} = \Sigma \bVect^\varsigma((\Z/2\oplus \Z/2 )[0]) \boxplus \Sigma \mathbf{Mod}_{\mathbf{2Vect}^\varsigma((\Z/2 \oplus \Z/2)[0]
    )}(\mathcal{B})\,.
\end{equation*}
This description does parse as a 3-category by noting that $\mathbf{Mod}_{\mathbf{2Vect}^\varsigma((\Z/2 \oplus \Z/2)[0])}(\mathcal{B})$ is a fusion 2-category \cite{DY22}. The data of $\alpha$ and $\tau$ twist the \textit{higher} coherence of the $\mathbb{Z}/2$-crossed braided structure. These twists are harder to see explicitly. We now give the fusion of two objects $\mathbf{D}^{(1)}_3$ and $\widetilde{\mathbf{D}}^{(1)}_3$ in the nontrivially graded component. By virtue of the fact that $\mathcal{B}$ has order two in $\Witt^{pt}(\mathbb{Z}/2\oplus \mathbb{Z}/2,s)$, there exists a $\Z/2 \oplus \Z/2$-graded fusion 1-category $\mathbf{D}^{(0)}_3$ such that $\mathcal{Z}(\mathbf{D}^{(0)}_3)= \mathcal{B} \underset{\Z/2 \oplus \Z/2} \boxtimes^{\varsigma} \mathcal{B}$ as $\Z/2 \oplus \Z/2$-graded braided fusion 1-categories. Then we have 
\be 
\mathbf{D}^{(1)}_3\times \widetilde{\mathbf{D}}^{(1)}_3 = (\mathbf{D}^{(1)}_3\underset{\Z/2 \oplus \Z/2} \boxtimes^{\varsigma}  \widetilde{\mathbf{D}}^{(1)}_3) \underset{(\mathcal{B}\underset{\Z/2 \oplus \Z/2} \boxtimes^{\varsigma}  \mathcal{B})} \boxtimes \mathbf{D}^{(0)}_3 \,,
\ee
where we think of $\mathbf{D}^{(0)}_3$ as playing the role of witnessing $\mathbf{D}^{(1)}_3 \times \widetilde{\mathbf{D}}^{(1)}_3$ being trivial in $\Witt^{pt}(\Z/2\oplus \Z/2,s)$. In particular, if we take the canonical objects in $\mathbf{Mod}_{\mathbf{2Vect}^\varsigma(\Z/2 \oplus \Z/2[0])}(\mathcal{B})$ given by $\mathcal{B}$, then 
\be 
\mathcal{B} \times \mathcal{B} = \mathbf{D}^{(0)}_3\,.
\ee
Picking any of the other elements of order 2 in $\Witt^{pt}(\mathbb{Z}/2\oplus \mathbb{Z}/2,s)$ lead to the same fusion rules for objects, and the categories can be related by using the automorphisms of the zero graded component.

\begin{Remark}
Recall that to construct the center of a 3TY, there is an equivariantization step as in Figure \ref{fig:gauging}. 
The underlying 3-category of $\mathscr{B}^{\Z/2}$, the $\Z/2$-equivariantization of $\mathscr{B}$, can be described as
\begin{equation*}
    \mathscr{B}^{\Z/2} = 
    \left(\Sigma\mathbf{2Vect}^\varsigma(\Z/2 \oplus \Z/2[0]
    )\right)^{\Z/2}\boxplus \left(\Sigma \mathbf{Mod}_{\mathbf{2Vect}^\varsigma((\Z/2 \oplus \Z/2)[0]
    )}(\mathcal{B})\right)^{\Z/2}\,.
\end{equation*}
The connected components of the $\Z/2$-equivariantization of $\mathscr{B}$, i.e.\ $\Omega \mathscr{E}^{\Z/2} $, is given by the sylleptic fusion 2-category that is the $\Z/2$-equivariantization of $\mathbf{2Vect}^\varsigma((\Z/2 \oplus \Z/2)[0])$. 
%\thib{I'm pretty sure that this is $\mathbf{2Rep}((\mathbb{Z}/2\oplus\mathbb{Z}/2)[1]\rtimes \mathbb{Z}/2[0])$ but with a modified syllepsis.} 
The $\Z/2$-equivariantization of $\Sigma \mathbf{Mod}_{\mathbf{2Vect}^\varsigma((\Z/2 \oplus \Z/2)[0]
)}(\mathcal{B})$ is equivalent to $\Sigma \mathfrak{C}$ where $\mathfrak{C}$ is a $\Z/2$-extension of the fusion 2-category $\mathbf{Mod}_{\mathbf{2Vect}^s((\Z/2 \oplus \Z/2)[0]
)}(\mathcal{B})$, but we were unable to identify this extension explicitly.

%In the description we have, it is possible to explicitly see the data on the defect in the SymTFT where the 0-form symmetry defect ends, and how that impacts $\cZ(\mathbf{3TY}[\Z/2])$. 
\end{Remark}

The previous discussion generalizes for other groups $G=\Z/K$ for $\Mod( \bVect^s(A))$ where $A$ is an abelian group. We consider a map $G\rightarrow \mathcal{W}itt(A,s)$, which corresponds to an element $\mathcal{B}\in \Witt(A,s)$ of order $K$. 
%The 3-category is given by 
%$$\mathscr{E}=\Sigma\big(\mathbf{2Vect}^s(E)\big)\boxplus \Sigma\big( \mathbf{Mod}_{\mathbf{2Vect}^s(E)}(\mathcal{B})\big)\boxplus \Sigma\big(\mathbf{Mod}_{\mathbf{2Vect}^s(E)}(\mathcal{B}\boxtimes^s_E\mathcal{B})\big)\boxplus ...$$
%where the product $\boxtimes^s_E$  is define in section \ref{subsection:twistedwitt}. We have the fusion rule

%\begin{tabular}{ccc}
%$\Sigma(\mathbf{Mod}_{\mathbf{2Vect}^s(E)}(\mathcal{B}))\times \Sigma(\mathbf{Mod}_{\mathbf{2Vect}^s(E)}(\mathcal{B}))$& $\rightarrow$ & $\Sigma(\mathbf{Mod}_{\mathbf{2Vect}^s(E)}(\mathcal{B}\boxtimes^s_E\mathcal{B})).$\\$ \mathbf{C}\times \mathbf{D}$&$\mapsto$ & $\mathbf{C}\boxtimes_E^s\mathbf{D}$
%\end{tabular}
%where $\mathbf{C}$ and $\mathbf{D}$ are $E$-graded fusion 1-categories. Fusing two objects which result in the zero-graded component requires trivialization data like in section \ref{subsub:examplesZ2}. This arises from taking $\Sigma(-)$ of the map

%\begin{tabular}{ccc}
%$\mathbf{Mod}_{\mathbf{2Vect}^s(E)}(\mathcal{B})\times \mathbf{Mod}_{\mathbf{2Vect}^s(E)}(\mathcal{B})$& $\rightarrow$ & $\mathbf{Mod}_{\mathbf{2Vect}^s(E)}(\mathcal{B}\boxtimes^s_E\mathcal{B})$.
%\end{tabular} \\

\section{Outlook and Challenges}\label{section:outlook}

In this work, we were concerned with the following mathematical objects having to do with extension theory for 3-categories: $\BrPic(\mathbf{3Vect}(A[1]))$, $\Pic(\cZ(\mathbf{3Vect}(A[1])))$, and $\Witt^{}(A \oplus A,\varsigma)$. While they are all related, it is hard to pass between each of the objects in a manifestly transparent way. A particular question that is natural to ask is the following: Given a twist defect in the bulk that decorates the ending of the 0-form symmetry
for the SymTFT, can one identify the precise duality defect on the boundary? This would mean understanding how a class $\widehat{\mathcal{B}}\in \mathcal{W}itt^{}(A\oplus {A},\varsigma)/\mathcal{W}itt^{}$, which has order two, maps to an element $M \in BrPic(\mathbf{3Vect}(A[1]))/\Witt^{}$. One can consider mapping to $\widehat{\mathcal{B}}\in \mathcal{W}itt^{}(A\oplus {A},\varsigma)/\mathcal{W}itt^{}$ by
passing through $Pic(\cZ(\cVect(A[1])))/\Witt^{}$ given the fact that there are equivalences
\be 
\frac{BrPic(\mathbf{3Vect}(A[1]))}{\Witt^{}} \longrightarrow \frac{Pic(\cZ(\cVect(A[1])))}{\Witt^{}} \longleftarrow \frac{\mathcal{W}itt^{}(A\oplus {A},s)}{\mathcal{W}itt^{}} \,.
\ee
A bimodule $\mathscr{M}$ in $BrPic(\mathbf{3Vect}(A[1]))/\Witt^{}$ is mapped to (the classes of)
\be 
Hom_{\cVect(A[1])\text{-}\cVect(A[1])}(\cVect(A[1]), \mathscr{M})\cong \cZ_{\cVect(A[1])}(\mathscr{M}) \,.
\ee
On the other hand, the map from $\Witt^{}(A\oplus A,s)/\Witt^{}$ to $Pic(\cZ(\cVect(A[1])))/\Witt^{}$ sends
\be 
\widehat{\cB} \mapsto \Sigma \Mod_{\cZ(\cVect(A[1]))}(\widehat{\cB}) \,.
\ee
In applications it would be useful to be able to describe the composite.

We also want to understand how an element in $Pic(\cZ(\mathbf{3Vect}(A[1])))$ acts on the center, i.e.\ which element of $Aut^{br}(\cZ(\mathbf{3Vect}(A[1])))$ it yields. This information is important in order to understand where the obstruction classes in section \ref{subsection:crossbraideddefs} lie, and goes beyond just computing the group structure of $Pic(\cZ(\mathbf{3Vect}(A[1])))$. Then, even in the case when the obstructions to a $G$-crossed braided extension of $\cZ(\mathbf{3Vect}(A[1]))$ vanish, we still must conduct the equivariantization step in Figure \ref{fig:gauging} for the gauging in order to recover an actual center, as in \cite{GNN} for 1-categories. Even for $\Z/2$-crossed braided extensions, the result of the equivariantization is ambiguous. For instance, we were not able to completely determine the $\Z/2$-equivariantization of $\Sigma\mathbf{Mod}_{\mathbf{2Vect}^\varsigma(\Z/2 \oplus \Z/2[0] )}(\mathcal{B})$ in the last section section \ref{subsub:examplesZ2}. A higher form of (de-)equivariantization is being developed by \cite{DHJFPPRNY} for the purpose of classifying fusion 2-categories and it remains to be seen if it can be helpful when performing computations.

%%%%%%%%%%%%%%%%%%%%%%%%%%%%%%%%%%%%%%%

\bibliographystyle{ytphys}
\bibliography{ref.bib}

%%%%%%%%%%%%%%%%%%%%%%%%%%%%%%%%%%%%%%%

\end{document}